\documentclass{amsart}
\usepackage{tikz,tikz-cd}
\usepackage{url}
\usepackage{graphicx}
\usepackage{color}
\usepackage{times}
\usepackage{cite}
\usepackage{enumitem,latexsym}
\usepackage{amsmath,amssymb}
\usepackage{amsthm}
\usepackage{verbatim}
\usepackage{mathrsfs}

\newtheorem{thm}{Theorem}[section]
\newtheorem*{theorem*}{Theorem}
\newtheorem*{acknowledgement*}{Acknowledgement}
\newtheorem{cor}[thm]{Corollary}

\newtheorem{lem}[thm]{Lemma}
\newtheorem{prop}[thm]{Proposition}
\theoremstyle{definition}

\theoremstyle{remark}
\newtheorem{rem}[thm]{Remark}

\numberwithin{equation}{section}

\newcommand{\set}[1]{\left\{#1\right\}}
\newcommand{\Real}{\mathbb R}

\newcommand{\func}[1]{\ensuremath{\operatorname{#1}} }

\newcommand{\Div}[0]{\func{div}}

\newcommand{\spt}[0]{\func{spt}}

\newcommand{\cC}[0]{\mathcal{C}}

\title[Relative expander entropy in the presence of a two-sided obstacle]{Relative expander entropy in the presence of a two-sided obstacle and applications}

\author{Jacob Bernstein}
\address{Department of Mathematics, Johns Hopkins University, 3400 N. Charles Street, Baltimore, MD 21218}
\email{bernstein@math.jhu.edu}

\author{Lu Wang}
\address{Department of Mathematics, California Institute of Technology, 1200 E. California Boulevard, Pasadena, CA 91125}
\email{drluwang@caltech.edu}

\begin{document}
	
\begin{abstract}
We study a notion of relative entropy motivated by self-expanders of mean curvature flow. In particular, we obtain the existence of this quantity for arbitrary hypersurfaces trapped between two self-expanders that are asymptotic to the same cone and bound a domain.  This allows us to begin to develop the variational theory for the relative entropy functional for the associated obstacle problem. We also obtain a version of the forward monotonicity formula for mean curvature flow proposed by Ilmanen.
\end{abstract}	
	
\maketitle

\section{Introduction} \label{IntroSec}
A \emph{hypersurface}, i.e., a properly embedded codimension-one submanifold, $\Sigma\subset\mathbb{R}^{n+1}$, is a \emph{self-expander} if
\begin{equation} \label{ExpanderEqn}
\mathbf{H}_\Sigma=\frac{\mathbf{x}^\perp}{2}.
\end{equation}
Here 
$$
\mathbf{H}_\Sigma=\Delta_\Sigma\mathbf{x}=-H_\Sigma\mathbf{n}_\Sigma=-\mathrm{div}_\Sigma(\mathbf{n}_\Sigma)\mathbf{n}_\Sigma
$$
is the mean curvature vector, $\mathbf{n}_\Sigma$ is the unit normal, and $\mathbf{x}^\perp$ is the normal component of the position vector. Self-expanders arise naturally in the study of mean curvature flow. Indeed, $\Sigma$ is a self-expander if and only if the associated family of homothetic hypersurfaces
$$
\left\{\Sigma_t\right\}_{t>0}=\left\{\sqrt{t}\, \Sigma\right\}_{t>0}
$$
is a \emph{mean curvature flow} (MCF). That is, a solution to 
$$
\left(\frac{\partial \mathbf{x}}{\partial t}\right)^\perp=\mathbf{H}_{\Sigma_t}.
$$
Given integers $k\geq 1$ and $n\geq 2$, $\Sigma$ is a \emph{$C^{k}$-asymptotically conical hypersurface in $\mathbb{R}^{n+1}$} with asymptotic cone $\cC=\mathcal{C}(\Sigma)$ if $\lim_{\rho\to 0^+} \rho \Sigma=\cC$ in $C^{k}_{loc} (\mathbb{R}^{n+1}\setminus\{\mathbf{0}\})$, where $\cC$ is a $C^k$-regular cone.  The space of such hypersurfaces is denoted by $\mathcal{ACH}^k_n$. If $\Sigma\in \mathcal{ACH}^k_n$ is a self-expander, then its associated flow emerges from $\cC(\Sigma)$ and so these self-expanders model how MCF resolves conical singularities.

Self-expanders are the critical points of the functional
$$
E[\Sigma]=\int_{\Sigma} e^{\frac{|\mathbf{x}|^2}{4}} d\mathcal{H}^n
$$
where $\mathcal{H}^n$ is $n$-dimensional Hausdorff measure. Due to the rapid growth of the weight this functional takes the value infinity on any asymptotically conical self-expander.  However, following a suggestion of Ilmanen \cite{IlmanenRel}, for $\Gamma_0, \Gamma_1\in \mathcal{ACH}^k_n$ with $\cC(\Gamma_0)=\cC(\Gamma_1)$ one may consider, when defined, the \emph{relative expander entropy}
$$
E_{rel}[\Gamma_1, \Gamma_0] =\lim_{R\to \infty} E_{rel}[\Gamma_1, \Gamma_0; \bar{B}_R]
$$
where
\begin{align*}
E_{rel}[\Gamma_1, \Gamma_0; \bar{B}_R]&= E[\Gamma_1\cap \bar{B}_R]-E[\Gamma_0\cap \bar{B}_R]\\
&=\int_{\Gamma_1\cap \bar{B}_R} e^{\frac{|\mathbf{x}|^2}{4}} d\mathcal{H}^n- \int_{\Gamma_0\cap \bar{B}_R} e^{\frac{|\mathbf{x}|^2}{4}} d\mathcal{H}^n.
\end{align*}
In the curve case, this relative functional was studied by Ilmanen-Neves-Schulze \cite{IlmanenNevesSchulze} who used it to prove the uniqueness of an expanding network in its topological class. More recently, Deruelle-Schulze \cite{DeruelleSchulze} investigated this relative functional in general dimensions and showed it is well defined and finite for pairs of self-expanders asymptotic to the same cone. Due to the rapid growth of the weight this is done by showing that the two self-expanders converge to each other at a very rapid rate -- see, for example, Proposition \ref{StrongDecayProp} below.  As a consequence, they are able to consider $E_{rel}$ as a sort of smooth function on the moduli space of self-expanders with varying cones -- by \cite{BWBanach}, this space has a natural manifold structure.  Their analysis allows them to conclude that $E_{rel}$ is non-zero on pairs of distinct self-expanders whose common asymptotic cone is generic in an appropriate sense. 

In this paper we develop the variational theory of the functional $E_{rel}$ in the presence of a natural two-sided obstacle. Among other things we show that $E_{rel}$ is well defined and coercive for arbitrary hypersurfaces satisfying the obstacle condition -- importantly, we achieve this without assuming any regularity at infinity for the hypersurfaces. More precisely, fix two self-expanders $\Gamma_0,\Gamma_1\in\mathcal{ACH}^2_n$ with $\cC(\Gamma_0)=\cC(\Gamma_1)=\cC$ and assume there are domains in $\mathbb{R}^{n+1}$, $U_0\subseteq U_1$ so that $\partial U_i=\Gamma_i$ for $i=0,1$. Let
$$
\mathcal{H}(\Gamma_0, \Gamma_1)=\set{\Gamma=\partial U \colon \mbox{$U$ is a smooth domain in $\mathbb{R}^{n+1}$ and $U_0\subseteq U\subseteq U_1$}}
$$
be the space of hypersurfaces trapped between $\Gamma_0$ and $ \Gamma_1$. While elements of $\mathcal{H}(\Gamma_0,\Gamma_1)$ are asymptotic to $\cC$ in the Hausdorff distance, in general there is no other asymptotic regularity.

We first show that the relative expander entropy $E_{rel}[\cdot, \Gamma_0]$ is well defined (possibly positive infinite) for all $\Gamma\in \mathcal{H}(\Gamma_0, \Gamma_1)$. 
 
\begin{thm}\label{RelEntropyMainThm}
If $\Gamma\in \mathcal{H}(\Gamma_0,  \Gamma_1)$, then 
 $$
E_{rel}[\Gamma, \Gamma_0]=\lim_{R\to \infty} E_{rel}[\Gamma,  \Gamma_0; \bar{B}_R] \in (-\infty, \infty].
$$
That is, the limit exists and is either real valued or positive infinity. 
\end{thm} 

\begin{rem} 
Some simple observations:
\begin{enumerate}
\item By \cite[Theorem 4.1]{BWUniqueExpander}, when $2\leq n \leq 6$, for every $C^3$-regular cone $\cC\subset \Real^{n+1}$, there are unique smooth domains $U_L\subseteq U_G$ satisfying $\Gamma_L=\partial U_L$ and $\Gamma_G=\partial U_G$ are self-expanders both $C^2$-asymptotic to $\cC$ and so that any asymptotically conical self-expander $\Gamma$ with $\cC(\Gamma)=\cC$ satisfies $\Gamma\in \mathcal{H}(\Gamma_L, \Gamma_G)$. Constructions of \cite{AIC} -- see also \cite{BWDegree} -- provide many examples where $\mathcal{H}(\Gamma_L, \Gamma_G)$ is non-trivial, i.e., it has more than one element.  
\item If $\Gamma\in \mathcal{H}(\Gamma_0, \Gamma_1)\cap \mathcal{ACH}^2_n$, i.e., $\Gamma$ is both trapped between $\Gamma_0$ and $\Gamma_1$ and $C^2$-asymptotic to $\cC$, then $E_{rel}[\Gamma, \Gamma_0]$ not only exists but is also finite -- see  Proposition \ref{RelEntropyAnnuliProp}. In this case the existence of $E_{rel}$ can be shown by adapting computations of Deruelle-Schulze \cite[Proposition 3.1]{DeruelleSchulze}.
\end{enumerate}
\end{rem}

It is useful to study an anisotropically weighted analog of $E_{rel}$. To describe the space of admissible weights, first fix a subset $W\subseteq\mathbb{R}^{n+1}$. For a function $\psi\in Lip(W\times \mathbb{S}^n)$ and any $p\in W$, define $\hat{\psi}_p(\mathbf{v})=\psi(p,\mathbf{v})$ and
$$
\nabla_{\mathbb{S}^n}\psi(p,\mathbf{v})=\nabla_{\mathbb{S}^n}\hat{\psi}_p(\mathbf{v}).
$$
Consider the Banach space
$$
\mathfrak{X}(W)=\set{\psi\in Lip(W\times\mathbb{S}^n) \colon \Vert \psi\Vert_{\mathfrak{X}}<\infty}
$$
where
$$
\Vert \psi \Vert_{\mathfrak{X}}=\Vert\psi\Vert_{Lip}+\Vert\nabla_{\mathbb{S}^n}\psi\Vert_{Lip}+\sup_{(p,\mathbf{v})\in W\times\mathbb{S}^n} (1+|\mathbf{x}(p)|)|\nabla_{\mathbb{S}^n}\psi(p,\mathbf{v})|.
$$
We let
$$
\mathfrak{X}^e(W)=\set{\psi\in \mathfrak{X}(W) \colon \psi(p,\mathbf{v})=\psi(p,-\mathbf{v}), \forall (p,\mathbf{v})\in W\times \mathbb{S}^n}.
$$
Elements of $\mathfrak{X}^e(W)$ are said to be \emph{even}. Observe that an even function is naturally identified with a function of the Grassman $n$-plane bundle of $W$.  

For $\Gamma\in \mathcal{H}(\Gamma_0, \Gamma_1)$ and $\psi \in \mathfrak{X}^e(\Real^{n+1})$, let
$$
E_{rel}[\Gamma, \Gamma_0; \psi; \bar{B}_R]=\int_{\Gamma\cap \bar{B}_R} \psi(p, \mathbf{n}_\Gamma(p)) e^{\frac{|\mathbf{x}|^2}{4}} d\mathcal{H}^n -\int_{\Gamma_0\cap \bar{B}_R} \psi(p, \mathbf{n}_{\Gamma_0}(p) )e^{\frac{|\mathbf{x}|^2}{4}}d\mathcal{H}^n, 
$$
and
$$
E_{rel}[\Gamma, \Gamma_0; \psi]=\lim_{R\to \infty} E_{rel}[\Gamma, \Gamma_0; \psi; \bar{B}_R]
$$
when this limit exists. Observe that if $\psi$ has compact support, then the limit is defined. We show that if $E_{rel}[\Gamma, \Gamma_0]$ is finite, then, for all $\psi\in\mathfrak{X}^e(\mathbb{R}^{n+1})$, $E_{rel}[\Gamma, \Gamma_0; \psi]$ exists and, moreover, the map $\psi \mapsto E_{rel}[\Gamma, \Gamma_0; \psi]$ is a bounded linear functional on $ \mathfrak{X}^e(\Real^{n+1})$.

\begin{thm}\label{WeightEstThm}
If $\Gamma\in \mathcal{H}(\Gamma_0,  \Gamma_1)$ has $E_{rel}[\Gamma, \Gamma_0]<\infty$, then, for any $\psi \in \mathfrak{X}^e(\Real^{n+1})$, $E_{rel}[\Gamma, \Gamma_0; \psi]$ exists. Moreover, there is a constant $L=L(\Gamma_0,\Gamma_1,n)\geq 0$ so that, for all $\psi\in \mathfrak{X}^e(\Real^{n+1})$,
\begin{align*}
\left| E_{rel}[\Gamma, \Gamma_0; \psi] \right| \leq & L (1+ |E_{rel}[\Gamma, \Gamma_0]|)\Vert \psi \Vert_{\mathfrak{X}}.
\end{align*}	
In particular, the map $\psi\mapsto E_{rel}[\Gamma, \Gamma_0; \psi]$ is a bounded linear functional on $\mathfrak{X}^e(\Real^{n+1})$.
\end{thm}

Theorems \ref{RelEntropyMainThm} and \ref{WeightEstThm} allow us to begin to develop the variational theory of $E_{rel}$ in  $\mathcal{H}(\Gamma_0,  \Gamma_1)$. In particular, in \cite{BWMinMax} a mountain pass theorem for $E_{rel}$ is proved.  In this paper we study the simpler question of minimizing $E_{rel}$ in $\mathcal{H}(\Gamma_0,  \Gamma_1)$. An element $\Gamma^\prime\in \mathcal{H}(\Gamma_0, \Gamma_1)$ is an \emph{$E_{rel}$-minimizer in  $\mathcal{H}(\Gamma_0, \Gamma_1)$} if, for all $\Gamma \in  \mathcal{H}(\Gamma_0, \Gamma_1)$, $E_{rel}[\Gamma, \Gamma_0]\geq E_{rel}[\Gamma^\prime, \Gamma_0].$  We directly establish the existence of $E_{rel}$-minimizers.

\begin{thm}\label{MinThm}
When $2\leq n \leq 6$, there exists a self-expander, $\Gamma_{min}$, that is an $E_{rel}$-minimizer  in  $\mathcal{H}(\Gamma_0, \Gamma_1)$. 
\end{thm}

\begin{rem}\label{MinimizerRemark}
It is worth comparing the notion of $E_{rel}$-minimizer with the more standard notion of a local {$E$-minimizer}. Recall, $\Gamma^\prime\in \mathcal{H}(\Gamma_0, \Gamma_1)$ is a \emph{local $E$-minimizer in $\mathcal{H}(\Gamma_0, \Gamma_1)$} provided $E[\Gamma\cap B_{R}]\geq E[\Gamma^\prime\cap B_R]$, for any $\Gamma\in  \mathcal{H}(\Gamma_0, \Gamma_1)$ that satisfies $\Gamma\backslash B_{R}=\Gamma'\backslash B_{R}$. Clearly, any $E_{rel}$-minimizer  in  $\mathcal{H}(\Gamma_0, \Gamma_1)$ is a local $E$-minimizer in  $\mathcal{H}(\Gamma_0, \Gamma_1)$.  As observed by Deruelle-Schulze \cite[Theorem 4.1]{DeruelleSchulze}, the converse is also true: a local $E$-minimizer in  $\mathcal{H}(\Gamma_0, \Gamma_1)$ is also an $E_{rel}$-minimizer in $\mathcal{H}(\Gamma_0, \Gamma_1)$. This is because their argument uses only that $E_{rel}$ is well defined and not $-\infty$ and a good estimate on the area of ribbons as in Lemma \ref{SliceAreaLem}.
\end{rem}

Another application is the existence of a forward monotonicity formula for mean curvature flows trapped between two disjoint expanders coming out of the same cone. This implies that any mean curvature flow that emerges from a cone and that is trapped between two self-expanders is initially modeled by a self-expander --  a fact used in \cite{BWIsotopy}. Related results for harmonic map flow were obtained previously by Deruelle \cite{Deruelle}.

\begin{thm} \label{MonotoneThm}
Let $\set{\Sigma_t}_{t\in (0, T)}$ be a mean curvature flow that satisfies
\begin{enumerate}
\item $\lim_{t\to 0} \mathcal{H}^n \lfloor \Sigma_t=\mathcal{H}^n \lfloor \cC$ for $\cC$ a $C^2$-regular cone;
\item For each $0<t<T$, $t^{-1/2}\Sigma_t\in \mathcal{H}(\Gamma_0,\Gamma_1)$.
\end{enumerate}
Then, for any sequence $t_i\to 0$, there is a subsequence $t_{i_j}\to 0$ so that
$$
t_{i_j}^{-1/2} \Sigma_{t_{i_j}}\to \Gamma
$$
where $\Gamma$ is a (possibly singular) self-expander $C^1$-asymptotic to $\mathcal{C}$ and the convergence is in the sense of measures.
\end{thm}

\begin{rem}
Ilmanen gave a sketch of the proof that the outermost flow from a cone is made up of stable self-expanders asymptotic to the cone -- see \cite[Lecture 2, {\bf F}]{IlmanenLec}. Thus, Hypothesis (2) of Theorem \ref{MonotoneThm} is expected to be unnecessary. 
\end{rem}

Finally, we remark that all of the above theorems also apply to lower regularity surfaces, specifically, to boundaries of Caccioppoli sets. They also apply to hypersurfaces trapped inside regions that are slightly ``thicker" than the one that lies between two ordered self-expanders.  Both of these more general situations are needed in applications and are treated in the body of the paper.

\subsection*{Acknowledgements}
The first author was partially supported by the NSF Grants DMS-1609340 and DMS-1904674 and the Institute for Advanced Study with funding provided by the Charles Simonyi Endowment. The second author was partially supported by the NSF Grants DMS-2018221(formerly DMS-1811144) and DMS-2018220 (formerly DMS-1834824), the funding from the Wisconsin Alumni Research Foundation and a Vilas Early Career Investigator Award by the University of Wisconsin-Madison, and a von Neumann Fellowship by the Institute for Advanced Study with funding from the Z\"{u}rich Insurance Company and the NSF Grant DMS-1638352. The authors would like to thank Alix Deruelle and Felix Schulze for their helpful comments.

\section{Notation and preliminaries}
We fix notation and certain conventions we will use throughout the remainder of the paper.  We also recall certain facts we will need.

\subsection{Basic notions} \label{NotionSec}
Denote a (open) ball in $\mathbb{R}^{n}$ centered at $p$ with radius $R$ by $B^{n}_R(p)$ and the closed ball by $\bar{B}^{n}_R(p)$. We often omit the superscript, $n$, when it is clear from context. We also omit the center when it is the origin. Likewise, denote an (open) annulus of inner radius $R_1$ and outer radius $R_2$ by $A_{R_1,R_2}$ and the closed annulus by $\bar{A}_{R_1,R_2}$.  We denote the closure of a set $U$ both by $\overline{U}$ and $\mathrm{cl}(U)$ and the topological boundary by $\partial U$.

Assume that $n,k\geq 2$ are integers. A \emph{cone} is a set $\cC\subseteq\mathbb{R}^{n+1}\setminus\{\mathbf{0}\}$ that is dilation invariant around the origin. That is, $\rho\mathcal{C}=\mathcal{C}$ for all $\rho>0$. The \emph{link} of the cone is the set $\mathcal{L}(\mathcal{C})=\cC\cap\mathbb{S}^{n}$, the intersection of the cone and the unit $n$-sphere. The cone is \emph{$C^{k}$-regular} if its link is an embedded, codimension-one, $C^{k}$ submanifold in $\mathbb{S}^n$. 

\subsection{Caccioppoli sets} \label{CaccioppoliSec}
Let $W$ be an open subset of $\mathbb{R}^{n+1}$. A subset $U\subseteq W$ is a \emph{Caccioppoli set} if it is a set of locally finite perimeter, that is $\mathbf{1}_U$, the characteristic function of $U$, belongs to $BV_{loc}(W)$. Given a Caccioppoli set $U$, let $\Gamma=\partial^* U$ be the reduced boundary of $U$ and let $\mathbf{n}_\Gamma$ be the outward unit normal to $U$. Without loss of generality, we assume $\mathrm{cl}(\partial^* U)=\partial U$ -- see \cite[Theorem 4.4]{Giusti}. 

For $i\in \set{0,1}$, let $U_i$ be Caccioppoli sets with $\Gamma_i=\partial^* U_i$. If $U_0\subseteq U_1$, then let
$$
\mathcal{C}(\Gamma_0,\Gamma_1)=\set{U\colon \mbox{$U$ is a Caccioppoli set and $U_0\subseteq U\subseteq U_1$}}.
$$
Let $\Omega=U_1\setminus \mathrm{cl}(U_0)$. Let $U$ be an element of $\mathcal{C}(\Gamma_0,\Gamma_1)$ and $\Gamma=\partial^* U$. For a function $\psi\in C^0_c(\overline{\Omega})$ define
$$
E[\Gamma,\Gamma_0;\psi]=\int_{\Gamma} \psi(p) e^{\frac{|\mathbf{x}(p)|^2}{4}} \, d\mathcal{H}^n-\int_{\Gamma_0} \psi(p) e^{\frac{|\mathbf{x}(p)|^2}{4}} \, d\mathcal{H}^n.
$$
More generally, for a function $\psi\in C^0_c (\overline{\Omega}\times \mathbb{S}^n)$ define
$$
E[\Gamma, \Gamma_0; \psi]=\int_{\Gamma} \psi(p, \mathbf{n}_\Gamma(p)) e^{\frac{|\mathbf{x}(p)|^2}{4}} \, d\mathcal{H}^n-\int_{\Gamma_0} \psi(p,\mathbf{n}_{\Gamma_0}(p)) e^{\frac{|\mathbf{x}(p)|^2}{4}} \, d\mathcal{H}^n.
$$
We remark that $E[\Gamma,\Gamma_0; \psi]$ is linear in $\psi$ and that when $\psi$ is {even} $E[\Gamma, \Gamma_0; \psi]$ is independent of the choice of $\mathbf{n}_{\Gamma}$ or $\mathbf{n}_{\Gamma_0}$.  

\subsection{Partial ordering of asymptotically conical hypersurfaces} \label{PartialOrderSec}
Let $\mathcal{C}$ be a $C^2$-regular cone in $\mathbb{R}^{n+1}$ so the link $\mathcal{L}(\mathcal{C})$ is an embedded codimension-one $C^2$ submanifold of $\mathbb{S}^n$. Clearly, $\mathcal{L}(\mathcal{C})$ separates $\mathbb{S}^n$ and we fix a closed set $\omega\subset\mathbb{S}^n$ so that $\partial\omega=\mathcal{L}(\mathcal{C})$. A hypersurface $\Sigma$ is \emph{asymptotic to $\mathcal{C}$} if 
$$
\lim_{\rho\to 0^+} \mathcal{H}^n\lfloor(\rho\Sigma)=\mathcal{H}^n\lfloor\mathcal{C}.
$$
When this occurs set $\mathcal{C}(\Sigma)=\mathcal{C}$. For such $\Sigma$, let $\Omega_-(\Sigma)$ be the subset of $\mathbb{R}^{n+1}\setminus\Sigma$ so that $W\cap\partial\Omega_-(\Sigma)=\Sigma$ and
$$
\lim_{\rho\to 0^+} \mathrm{cl}(\rho \Omega_-(\Sigma))\cap\mathbb{S}^n=\omega \mbox{ as closed sets}.
$$
Such $\Omega_-(\Sigma)$ is well defined by the hypotheses on $\Sigma$. Denote by $\Omega_+(\Sigma)=\mathbb{R}^{n+1}\setminus \overline{\Omega_-(\Sigma)}$. For hypersurfaces $\Sigma_0, \Sigma_1$ for which  $\mathcal{C}(\Sigma_0)=\mathcal{C}(\Sigma_1)$ write 
$$
\Sigma_0\preceq \Sigma_1 \mbox{ provided } \Omega_-(\Sigma_0) \subseteq \Omega_-(\Sigma_1).
$$

\subsection{Conventions}\label{Conventions}
We now fix conventions we will use in the remainder of the paper.  Let $\mathcal{C}$ be a $C^2$-regular cone in $\mathbb{R}^{n+1}$. Pick a closed set $\omega\subset\mathbb{S}^n$ so $\partial\omega=\mathcal{L}(\cC)$. Using $\omega$, let $\Gamma_0,\Gamma_1$ be two self-expanders both $C^2$-asymptotic to $\mathcal{C}$ and assume $\Gamma_0\preceq\Gamma_1$. Denote by $\Omega=\Omega_+(\Gamma_0)\cap\Omega_-(\Sigma_1)$. Let $\nabla$,  $\mathrm{div}$ and $\Delta$ denote, respectively, the gradient, the divergence and the Laplacian on $\mathbb{R}^{n+1}$.

If $\Gamma$ is a $C^2$-asymptotically conical self-expander in $\mathbb{R}^{n+1}$, then it follows from the interior estimates for MCF (see, e.g., Theorem 3.4 and Remark 3.6 (ii) of \cite{EHInterior}) that 
\begin{equation} \label{LinearDecayEqn}
C_{\Gamma,l}=\sup_{p\in\Gamma} \left((1+|\mathbf{x}(p)|) \sum_{i=1}^l |\nabla^i_\Gamma \mathbf{n}_{\Gamma}(p)| \right)<\infty.
\end{equation}

We also introduce the following test functions. Let
$$
\phi_{R,\delta}(p)=\left\{\begin{array}{ll} 1 & \mbox{if $p\in B_R$} \\ 1-\frac{|\mathbf{x}(p)|-R}{\delta} & \mbox{if $p\in\bar{A}_{R,R+\delta}$} \\
0 & \mbox{if $p\in\mathbb{R}^{n+1}\setminus\bar{B}_{R+\delta}$}\end{array} \right.
$$
be a cutoff. Let
$$
\alpha_{R_1, R_2, \delta}(p)= \phi_{R_2,\delta}(p)-\phi_{R_1-\delta, \delta}(p) \in Lip_c(\Real^{n+1})
$$
be the cutoff adapted to the closed annulus $\bar{A}_{R_1,R_2}$.

Finally recall that a set $Y\subset\mathbb{R}^{n+1}$ is \emph{quasi-convex} if there is a constant $C>0$ so that any pair of points $p,q\in Y$ can be joined by a curve $\beta$ in $Y$ with 
$$
\mathrm{Length}(\beta)\leq C|\mathbf{x}(p)-\mathbf{x}(q)|.
$$ 
It is readily checked that $\overline{\Omega}$ and $\overline{\Omega}\setminus\bar{B}_R$ are both quasi-convex and so, by \cite[Theorem 4.1]{Heinonen}, the space of Lipschitz functions on these domains is the same as the $W^{1,\infty}$ space. 

\subsection{Decay estimates for self-expanding ends and an area estimate} \label{DecayAreaSec}

Using estimates of the first author \cite{Bernstein} -- cf. \cite[Theorem 2.1]{DeruelleSchulze} -- one obtains strong asymptotic decay results for the ends of two expanders asymptotic to the same cone. We will use this in order to obtain sharp area estimates for the slices of large spheres lying between two ordered expanders asymptotic to the same cone.

\begin{prop}\label{StrongDecayProp}
Let $\mathcal{C}$ be a $C^2$-regular cone in $\mathbb{R}^{n+1}$. Suppose $\Sigma_0$ and $\Sigma_1$ are self-expanding ends both $C^2$-asymptotic to $\mathcal{C}$. There is a radius $\bar{R}_0=\bar{R}_0(\Sigma_0,\Sigma_1)>1$ and a constant $\bar{C}_0=\bar{C}_0(\Sigma_0,\Sigma_1)>0$ so that there is a smooth function $u\colon\Sigma_0\setminus\bar{B}_{\bar{R}_0}\to\mathbb{R}$ satisfying
$$
\Sigma_1\setminus\bar{B}_{2\bar{R}_0}\subset\set{\mathbf{x}(p)+u(p)\mathbf{n}_{\Sigma_0}(p)\colon p\in\Sigma_0\setminus\bar{B}_{\bar{R}_0}}\subset\Sigma_1
$$
and $u$ satisfies the (sharp) pointwise estimate
$$
|u|+r^{-1} |\nabla_{\Sigma_0} u| + r^{-2} |\nabla^2_{\Sigma_0} u|\leq \bar{C}_0 r^{-n-1} e^{-\frac{r^2}{4}}
$$
where $r(p)=|\mathbf{x}(p)|$ for $p\in\Sigma_0$.  Moreover, for any $R>2\bar{R}_0$, 
$$
\Sigma_1\backslash B_{R}\subset \mathcal{T}_{\bar{C}_0 R^{-n-1} e^{-\frac{R^2}{4}}} (\Sigma_0).
$$
Here $\mathcal{T}_\delta (\Sigma_0)$ is the $\delta$-tubular neighborhood of $\Sigma_0$.
\end{prop}

To prove Proposition \ref{StrongDecayProp} we need a couple of auxiliary lemmas which, due to their technical nature, are collected in Appendix \ref{AuxiliaryApp}.

\begin{proof}[Proof of Proposition \ref{StrongDecayProp}]
As $\Sigma_0$ and $\Sigma_1$ are both $C^2$-asymptotic to $\mathcal{C}$, it follows from \cite[Proposition 3.3]{BWProperness} that there are constants $\mathcal{R}=\mathcal{R}(\Sigma_0,\Sigma_1)>1$ and $M=M(\Sigma_0,\Sigma_1)>0$, and functions $f_0$ and $f_1$ on $\mathcal{C}\setminus\bar{B}_{\mathcal{R}}$ so that, for $i\in\set{0,1}$,
$$
\Sigma_i\setminus\bar{B}_{2\mathcal{R}}\subset\set{\mathbf{f}_i(\bar{p})=\mathbf{x}(\bar{p})+f_i(\bar{p})\mathbf{n}_{\mathcal{C}}(\bar{p})\colon \bar{p}\in\mathcal{C}\setminus\bar{B}_{\mathcal{R}}}\subset\Sigma_i
$$
with the curvature estimate
$$
\sup_{p\in\Sigma_i\setminus\bar{B}_{2\mathcal{R}}} |\mathbf{x}(p)| |A_{\Sigma_i}(p)| \leq M,
$$
and $f_i$ satisfies
$$
|f_i(\bar{p})|+|\nabla_{\mathcal{C}} f_i(\bar{p})| \leq M |\mathbf{x}(\bar{p})|^{-1} \leq M\mathcal{R}^{-1} \leq \frac{1}{2}.
$$
By the triangle inequality 
$$
\frac{1}{2} |\mathbf{x}(\bar{p})| \leq |\mathbf{f}_i(\bar{p})| \leq 2 |\mathbf{x}(\bar{p})|.
$$

As 
$$
|\mathbf{f}_1(\bar{p})-\mathbf{f}_0(\bar{p})| \leq |f_1(\bar{p})|+|f_0(\bar{p})| \leq 2M |\mathbf{x}(\bar{p})|^{-1} \leq 4M |\mathbf{f}_1(\bar{p})|^{-1}
$$
it follows that 
$$
\mathrm{dist}(\mathbf{f}_1(\bar{p}),\Sigma_0) \leq |\mathbf{f}_1(\bar{p})-\mathbf{f}_0(\bar{p})|\leq  4M |\mathbf{f}_1(\bar{p})|^{-1}.
$$
Thus, for all $q\in\Sigma_1\setminus\bar{B}_{2\mathcal{R}}$,
$$
|\mathbf{x}(q)-\Pi_{\Sigma_0}(q)| \leq 4M |\mathbf{x}(q)|^{-1}
$$
where $\Pi_{\Sigma_0}$ is the nearest point projection to $\Sigma_0$. By our choice of $M$ and the triangle inequality, if $q\in\Sigma_1\setminus\bar{B}_{4\mathcal{R}}$, then
$$
\frac{1}{2} |\mathbf{x}(q)| \leq |\Pi_{\Sigma_0}(q)| \leq 2 |\mathbf{x}(q)|
$$
and, hence, 
\begin{equation} \label{DistanceEqn}
|\mathbf{x}(q)-\Pi_{\Sigma_0}(q)| \leq 8M |\Pi_{\Sigma_0}(q)|^{-1}.
\end{equation}

Given $q\in\Sigma_1\setminus\bar{B}_{16\mathcal{R}}$, suppose $\mathbf{x}(q)=\mathbf{f}_1(\bar{p})$ for some $\bar{p}\in\mathcal{C}\setminus\bar{B}_{\mathcal{R}}$. By the previous estimates and the triangle inequality
\begin{align*}
|\mathbf{f}_0(\bar{p})-\Pi_{\Sigma_0}(q)| & \leq |\mathbf{f}_0(\bar{p})-\mathbf{f}_1(\bar{p})|+|\mathbf{x}(q)-\Pi_{\Sigma_0}(q)| \\
& \leq 4M |\mathbf{x}(q)|^{-1}+8M |\Pi_{\Sigma_0}(q)|^{-1} \leq 16 M |\Pi_{\Sigma_0}(q)|^{-1}.
\end{align*}
In particular, $|\mathbf{f}_0(\bar{p})| \geq \frac{1}{2} |\Pi_{\Sigma_0}(q)|$ and so both $\mathbf{f}_0(\bar{p})$ and $|\Pi_{\Sigma_0}(q)|$ are in $\Sigma_0\setminus \bar{B}_{4\mathcal{R}}$. By the curvature decay of $\Sigma_0$ and enlarging $\mathcal{R}$, if needed, one has
$$
d_{\Sigma_0}(\mathbf{f}_0(\bar{p}),\Pi_{\Sigma_0}(q)) \leq 2|\mathbf{f}_0(\bar{p})-\Pi_{\Sigma_0}(q)| \leq 32 M |\Pi_{\Sigma_0}(q)|^{-1}
$$
and so 
$$
|\mathbf{n}_{\Sigma_0}(\mathbf{f}_0(\bar{p}))-\mathbf{n}_{\Sigma_0}(\Pi_{\Sigma_0}(q))| \leq 8 M |\Pi_{\Sigma_0}(q)|^{-1}.
$$
One also uses the $C^1$ bound for $f_i$ to get
$$
|\mathbf{n}_{\Sigma_1}(\mathbf{f}_1(\bar{p}))-\mathbf{n}_{\Sigma_0}(\mathbf{f}_0(\bar{p}))| \leq C M |\mathbf{x}(\bar{p})|^{-1} \leq 2CM |\Pi_{\Sigma_0}(q)|^{-1} 
$$
for some $C=C(n)$. Thus, combining these two estimates gives 
\begin{equation} \label{DiffNormalEqn}
|\mathbf{n}_{\Sigma_1}(q)-\mathbf{n}_{\Sigma_0}(\Pi_{\Sigma_0}(q))| \leq 2(C+4) M |\Pi_{\Sigma_0}(q)|^{-1}.
\end{equation}

Hence, in view of \eqref{DistanceEqn} and \eqref{DiffNormalEqn}, there are constants $\bar{R}_0=\bar{R}_0(n,C,M,\mathcal{R})>1$ and $\bar{M}=\bar{M}(n,C,M)>0$ (which, in turn, depend only on $\Sigma_0$ and $\Sigma_1$) and a function $u\colon\Sigma_0\setminus\bar{B}_{\bar{R}_0}\to\mathbb{R}$ so that 
$$
\Sigma_1\setminus\bar{B}_{2\bar{R}_0} \subset\set{\mathbf{x}(p)+u(p)\mathbf{n}_{\Sigma_0}(p)\colon p\in\Sigma_0\setminus\bar{B}_{\bar{R}_0}} \subset\Sigma_1
$$
and $u$ satisfies the pointwise estimate
$$
|u(p)|+|\nabla_{\Sigma_0} u(p)|+|\nabla_{\Sigma_0}^2 u(p)| \leq \bar{M} |\mathbf{x}(p)|^{-1} \leq \frac{1}{2}.
$$
This together with Lemma \ref{ExpanderMeanCurvLem} implies that 
$$
L_{\Sigma_0} u=\Delta_{\Sigma_0} u+\frac{\mathbf{x}}{2}\cdot\nabla_{\Sigma_0} u+\left(|A_{\Sigma_0}|^2+\frac{1}{2}\right) u=\mathbf{a}\cdot\nabla_{\Sigma_0} u+bu
$$
and
$$
|\mathbf{a}|+|b| \leq \bar{C}_1 \left(|u|+|\nabla_{\Sigma_0} u|+|\mathbf{x}\cdot\nabla_{\Sigma_0} u|+|\nabla_{\Sigma_0}^2 u|\right) \leq \bar{C}_1(1+\bar{M})
$$
where $\bar{C}_1=\bar{C}_1(n,\Sigma_0)>0$. Thus, write
$$
\frac{\mathbf{x}}{2}\cdot\nabla_{\Sigma_0} u=-\Delta_{\Sigma_0} u-|A_{\Sigma_0}|^2 u+\frac{1}{2} u+\mathbf{a}\cdot\nabla_{\Sigma_0} u+bu
$$
and so, by the curvature decay of $\Sigma_0$ and estimates on $u$ and $|\mathbf{a}|+|b|$, one gets that $|\mathbf{x}\cdot\nabla_{\Sigma_0} u|$ decays linearly and, hence, so does $|\mathbf{a}|+|b|$. As such, one uses \cite[Theorem 9.1]{Bernstein} to see
$$
\int_{\Sigma_0\setminus\bar{B}_{\bar{R}_0}} u^2 e^{\frac{r^2}{8}} \, d\mathcal{H}^n <\infty.
$$
Hence, by the $L^\infty$ estimate \cite[Theorem 8.17]{GilbargTrudinger} and the Schauder estimate \cite[Theorem 6.2]{GilbargTrudinger}, one has that $|u|, |\nabla_{\Sigma_0} u|$ and $|\nabla_{\Sigma_0}^2 u|$ all decay faster than $e^{-\frac{1}{32}r^2}$ and so the same holds true for $|\mathbf{a}|$ and $b$. 
	
On $\Sigma_0\backslash \bar{B}_{\bar{R}_0}$ consider the barrier
$$
\varphi=r^{-n-1} e^{-\frac{r^2}{4}}-r^{-n-2} e^{-\frac{r^2}{4}}\leq r^{-n-1} e^{-\frac{r^2}{4}}.
$$ 
By increasing $\bar{R}_0$, if necessary, one may ensure $\varphi>0$. Moreover, using Lemma \ref{BarrierEqnLem}, one readily evaluates that, up to increasing $\bar{R}_0$ in a way that depends only on $\Sigma_0$ and $u$, 
$$
L_{\Sigma_0} \varphi \leq \mathbf{a}\cdot\nabla_{\Sigma_0} \varphi +b \varphi.
$$
Pick $\gamma>1$ large enough so that $|u| \leq \gamma \varphi$ on $\Sigma_0\cap \partial B_{\bar{R}_0}$. As $\varphi$ and $u$ both tend to $0$ as $r\to \infty$ and, up to further increasing $\bar{R}_0$,  $|A_{\Sigma_0}|^2-\frac{1}{2}-b<0$, it follows from the maximum principle that 
$$
|u| \leq \gamma \varphi \mbox{ on $\Sigma_0\setminus\bar{B}_{\bar{R}_0}$.}
$$
The pointwise estimate on derivatives of $u$ follow from standard Schauder estimates on balls for an appropriate choice of $\bar{C}_0$ -- see \cite[Corollary 4.12]{CIM} for the idea.

To complete the proof observe that when $R>2\bar{R}_0$ if $q\in \Sigma_1\backslash B_R$, then 
$$
|\mathbf{x}(q)-\Pi_{\Sigma_0}(q)|\leq 2\bar{M} |\mathbf{x}(q)|^{-1}\leq 2\bar{M} R^{-1}.
$$
Thus, 
$$
|\Pi_{\Sigma_0}(q)| \geq R-2\bar{M} R^{-1} > \frac{1}{2} R.
$$
One readily checks that 
$$
|u(\Pi_{\Sigma_0}(q))| \leq \gamma\varphi(\Pi_{\Sigma_0}(q)) \leq 2^{n+1} \gamma e^{\bar{M}} R^{-n-1} e^{-\frac{R^2}{4}}.
$$
Hence, as long as one chooses $\bar{C}_0\geq 2^{n+1}\gamma e^{\bar{M}}$, one has
$$
\Sigma_1\backslash B_{R}\subset \mathcal{T}_{\bar{C}_0R^{-n-1} e^{-\frac{R^2}{4}}}(\Sigma_0) 
$$
and this proves the final claim.
\end{proof}

An immediate consequence of Proposition \ref{StrongDecayProp} is that if $\Gamma_0$ and $\Gamma_1$ are two asymptotically conical self-expanders with $\cC(\Gamma_0)=\cC(\Gamma_1)$ and $\Gamma_0\preceq \Gamma_1$, then, for any $R>2\bar{R}_0$, the region $\Omega=\Omega_+(\Gamma_0)\cap \Omega_-(\Gamma_1)$ satisfies 
$$
\Omega \backslash B_R \subset \mathcal{T}_{\bar{C}_0 R^{-n-1} e^{-\frac{R^2}{4}}}(\Gamma_0)
$$
The above result implies that $\Omega$, the region between the two self-expanders $\Gamma_0$ and $\Gamma_1$ is ``thin" near infinity. For technical reasons important in later applications, it is useful to consider slight ``thickenings" of $\Omega$ that are still thin at infinity in this sense.  

More precisely, let $\Gamma_0^\prime$ and $\Gamma_1^\prime$ be two asymptotically conical hypersurfaces, not necessarily self-expanders, with $\cC(\Gamma_0^\prime)=\cC(\Gamma_1^\prime)=\cC(\Gamma_0)=\cC$ and so that $\Gamma_0^\prime\preceq \Gamma_0\preceq \Gamma_1^\prime$. Observe that if, in addition, $\Gamma_1 \preceq \Gamma_1^\prime$, then $\mathcal{C}(\Gamma_0,\Gamma_1)\subseteq \mathcal{C}(\Gamma_0^\prime, \Gamma_1^\prime)$. Let $\Omega^\prime=\Omega_-(\Gamma_1^\prime)\cap \Omega_+(\Gamma_0^\prime)$. The set $\Omega^\prime$ is \emph{thin at infinity relative to $\Gamma_0$} if it is quasi-convex and there are constants $\bar{C}^\prime_0=C^\prime_0(\Omega^\prime,\Gamma_0)>0$ and $\bar{R}_0^\prime=\bar{R}_0^\prime(\Omega^\prime,\Gamma_0)>1$ so that, for all $R>\bar{R}_0^\prime$,
\begin{equation}
\Omega^\prime\backslash B_R \subset \mathcal{T}_{\bar{C}_0^\prime R^{-n-1}e^{-\frac{R^2}{4}}} (\Gamma_0).
\end{equation}

Being thin at infinity may be thought of as a $C^0$ notion of ``thinness". Our arguments will mostly rely on a different notion of thinness related to the area of the ribbon sliced out by the region inside spheres. This is a weaker condition than thin at infinity.

\begin{lem} \label{SliceAreaLem}
If $\Omega^\prime$ is thin at infinity relative to $\Gamma_0$, then there is a constant $\bar{C}_1^\prime=\bar{C}_1^\prime(\Omega^\prime,\Gamma_0)$ so that, for all $R>0$, 
$$
\mathcal{H}^n(\Omega^\prime\cap \partial B_R)\leq  \bar{C}_1^\prime R^{-2} e^{-\frac{R^2}{4}}.
$$
\end{lem}

\begin{proof}
Let $\Pi_{\Gamma_0}$ be the nearest point projection to $\Gamma_0$. As $\Omega^\prime$ is thin at infinity, for any $q\in \Omega^\prime\setminus\bar{B}_{\bar{R}_0^\prime}$,
$$
|\mathbf{x}(q)-\Pi_{\Gamma_0}(q)| \leq \bar{C}_0^\prime |\mathbf{x}(q)|^{-n-1} e^{-\frac{|\mathbf{x}(q)|^2}{4}}.
$$
Choosing $\mathcal{R}_0>\max\set{\bar{R}_0^\prime, 2\bar{C}_0^\prime}+1$ if $q\in\Omega^\prime\setminus B_{\mathcal{R}_0}$, then
$$
\frac{1}{2}|\mathbf{x}(q)|< |\mathbf{x}(q)|-\bar{C}_0^\prime |\mathbf{x}(q)|^{-1} \leq |\Pi_{\Gamma_0}(q)| \leq |\mathbf{x}(q)|+\bar{C}_0^\prime |\mathbf{x}(q)|^{-1} < 2|\mathbf{x}(q)|
$$
and so 
$$
|\mathbf{x}(q)-\Pi_{\Gamma_0}(q)| < 2^{n+1}\bar{C}_0^\prime e^{\bar{C}_0^\prime} |\Pi_{\Gamma_0}(q)|^{-n-1} e^{-\frac{|\Pi_{\Gamma_0}(q)|^2}{4}}.
$$
Set 
$$
\varphi(p)=2^{n+1}\bar{C}_0^\prime e^{\bar{C}_0^\prime} |\mathbf{x}(p)|^{-n-1} e^{-\frac{|\mathbf{x}(p)|^2}{4}}
$$
and let 
$$
\Omega_{\varphi}=\set{\mathbf{x}(p)+t \mathbf{n}_{\Gamma_0}(p) \colon p\in \Gamma_0, |t|\leq\varphi(p)}.
$$
Thus one has
$$
\Omega^\prime\setminus B_{\mathcal{R}_0}\subset\Omega_\varphi\setminus B_{\mathcal{R}_0}.
$$
Moreover, up to increasing $\mathcal{R}_0$ in a way that depends only on $n$ and $\bar{C}_0^\prime$ one can ensure that, for all $q\in\Omega_\varphi\setminus B_{\mathcal{R}_0}$,
$$
\frac{1}{2} |\mathbf{x}(q)| \leq |\Pi_{\Gamma_0}(q)| \leq 2 |\mathbf{x}(q)|.
$$

Let
$$
\Gamma^\pm=\set{\mathbf{x}(p)\pm \varphi(p) \mathbf{n}_{\Gamma_0}(p) \colon p\in \Gamma_0}.
$$
Observe that for $R$ sufficiently large $\Gamma_R^\pm=\Gamma^\pm\backslash \bar{B}_R$ are both asymptotically conical hypersurfaces in $\Real^{n+1}\backslash \bar{B}_{R}$. Let $\delta_0=\delta_0(\Gamma_0)$ be the constant given by Proposition \ref{VolumeEstProp}. Thus, there is a radius $\mathcal{R}_1=\mathcal{R}_1(\Gamma_0, \bar{C}_0^\prime, \delta_0)>\mathcal{R}_0$ and functions $\theta^\pm$ on $\Gamma_0\setminus\bar{B}_{\mathcal{R}_1}$ so that 
$$
\Gamma^\pm_{\mathcal{R}_1}=\set{\mathbf{f}^\pm(p)=\cos\theta^\pm(p)\mathbf{x}(p)+|\mathbf{x}(p)|\sin\theta^\pm (p)\nu_{\Gamma_0}(p)\colon p\in\Gamma_0\setminus\bar{B}_{\mathcal{R}_1}}
$$
where $\nu_{\Gamma_0}(p)$ is the unit normal (in $\partial B_{|\mathbf{x}(p)|}$) to $\Gamma_0\cap\partial B_{|\mathbf{x}(p)|}$ at $p$, and $\theta^\pm$ satisfy
$$
\sup_{p\in\Gamma_0\setminus\bar{B}_{\mathcal{R}_1}} \left( |\theta^\pm(p)|+|\mathbf{x}(p)||\nabla_{\Gamma_0}\theta^\pm(p)| \right) \leq \delta_0.
$$
Let $\hat{\Pi}_{\Gamma_0}(\mathbf{y})$ be the nearest point projection (in $\partial B_{|\mathbf{y}|}$) of $\mathbf{y}$ to $\Gamma_0\cap \partial B_{|\mathbf{y}|}$. Up to increasing  $\mathcal{R}_1$, $\hat{\Pi}_{\Gamma_0}$ restricts to a $C^1$ map from $\Omega_\varphi\setminus \bar{B}_{\mathcal{R}_1}$ to $\Gamma_0$ with its gradient bound by $C>1$. 

If $\mathbf{h}^\pm=\Pi_{\Gamma_0}\circ\mathbf{f}^\pm$, then one readily checks that, for any $p\in\Gamma_0\setminus\bar{B}_{2\mathcal{R}_1}$,
\begin{align*}
|\mathbf{h}^\pm(p)-\mathbf{x}(p)| & = | \hat{\Pi}_{\Gamma_0}(\mathbf{f}^\pm(p))-\hat{\Pi}_{\Gamma_0}(\mathbf{h}^\pm(p))| \\
& \leq \Vert \nabla \hat{\Pi}_{\Gamma_0} \Vert_{C^0} | \mathbf{h}^\pm(p)-\mathbf{f}^\pm(p) | \leq C \varphi(\mathbf{h}^\pm(p)).
\end{align*}
By increasing $\mathcal{R}_1$ in a way that depends on $\bar{C}_0^\prime$ and $C$, this gives that
$$
|\varphi(\mathbf{h}^\pm(p))-\varphi(p)|  \leq  \Vert\nabla \varphi\Vert_{C^0} |\mathbf{h}^\pm(p)-\mathbf{x}(p)|<\frac{1}{2} \varphi(\mathbf{h}^\pm(p))
$$
and so $\varphi(\mathbf{h}^\pm(p)) \leq 2 \varphi(p)$. Thus, using these estimates one computes, on $\Gamma_0\setminus\bar{B}_{2\mathcal{R}_1}$,
\begin{align*}
|\mathbf{x}(p)| |\sin\theta^\pm(p) |& =|(\mathbf{h}^\pm(p)\pm \varphi(\mathbf{h}^\pm (p))\mathbf{n}_{\Gamma_0}(\mathbf{h}^\pm(p)))\cdot\nu_{\Gamma_0}(p)| \\
& \leq |(\mathbf{h}^\pm(p)-\mathbf{x}(p))\cdot\nu_{\Gamma_0}(p)|+ \varphi(\mathbf{h}^\pm(p))|\mathbf{n}_{\Gamma_0}(\mathbf{h}^\pm(p))\cdot\nu_{\Gamma_0}(p)| \\
& \leq |\mathbf{h}^\pm(p)-\mathbf{x}(p)|+\varphi(\mathbf{h}^\pm(p)) \\
& \leq 2(C+1) \varphi(p).
\end{align*}
In particular, $|\sin\theta^\pm(p)|<\frac{3}{10}$ so $|\theta^\pm(p)|\leq 2|\sin\theta^\pm (p)|$. Hence one has, on $\Gamma_0\setminus\bar{B}_{2\mathcal{R}_1}$,
$$
|\theta^\pm(p)| \leq 4(C+1)|\mathbf{x}(p)|^{-1} \varphi(p).
$$
	
It follows from Proposition \ref{VolumeEstProp} that, for all $R>2\mathcal{R}_1$,
$$
\mathcal{H}^n(\Omega_\varphi\cap \partial B_R) \leq 8(C+1) \int_{\Gamma_0\cap \partial B_R} \varphi \, d\mathcal{H}^{n-1}.
$$
As $\Gamma_0$ is asymptotic to $\mathcal{C}$, up to increasing $\mathcal{R}_1$, for all $R>2\mathcal{R}_1$,
$$
\mathcal{H}^{n-1}(\Gamma_0\cap\partial B_R) \leq 2R^{n-1}\mathcal{H}^{n-1}(\mathcal{L}(\mathcal{C})).
$$
Hence, for all $R>2\mathcal{R}_1$,
$$
\mathcal{H}^n(\Omega_\varphi\cap \partial B_R) \leq 16(C+1) \bar{C}_0^\prime \mathcal{H}^{n-1}(\mathcal{L}(\cC)) R^{-2} e^{-\frac{R^2}{4}}.
$$

As remarked before, for all $R>2\mathcal{R}_1$, 
$$
\Omega^\prime\cap \partial B_R \subset \Omega_\varphi \cap \partial B_R
$$
and, hence, 
$$
\mathcal{H}^n(\Omega^\prime\cap\partial B_R) \leq 16(C+1) \bar{C}_0^\prime \mathcal{H}^{n-1}(\mathcal{L}(\cC)) R^{-2} e^{-\frac{R^2}{4}}.
$$
The result follows for $R>2\mathcal{R}_1$ as long as 
$$
\bar{C}_1^\prime>16(C+1)\bar{C}_0^\prime\mathcal{H}^{n-1}(\mathcal{L}(\cC)).
$$
As $\mathcal{R}_1$ depends only on $\Gamma_0$ and $\Omega^\prime$, the result automatically holds for $R\leq  2\mathcal{R}_1$ as long as one chooses $\bar{C}_1^\prime$ sufficiently large.
\end{proof}

\section{Relative expander entropy} \label{CalibrationSec}
In this section we prove that the relative entropy for singular hypersurfaces, i.e., reduced boundaries of Caccioppoli sets, that lie within an asymptotically ``thin" set is well defined and not $-\infty$.  To that end we always take $\Gamma_0$ to be an asymptotically conical self-expander and $\Gamma_0^\prime, \Gamma_1^\prime$ be asymptotical conical hypsersurfaces so $\Gamma_0^\prime\preceq \Gamma_0 \preceq \Gamma_1 \preceq \Gamma_1^\prime$ and so $\Omega^\prime=\Omega_+(\Gamma_0^\prime)\cap \Omega_-(\Gamma_1^\prime)$ is thin at infinity relative to $\Gamma_0$ with constants $\bar{C}_0^\prime=\bar{C}_0^\prime(\Omega^\prime,\Gamma_0)$ and $\bar{R}_0^\prime=\bar{R}_0^\prime(\Omega^\prime,\Gamma_0)$ given in the definition. In addition to these conventions and those adopted in Section \ref{Conventions}, we also  will always take $\Gamma=\partial^* U$ for some $U\in \mathcal{C}(\Gamma_0^\prime,\Gamma_1^\prime)$. 

\begin{thm}\label{RelEntropyThm}
If $R_2>R_1>R_0$, then
$$
E_{rel}[\Gamma, \Gamma_0; \bar{B}_{R_2}]\geq E_{rel}[\Gamma, \Gamma_0; \bar{B}_{R_1}]-C_2 R_1^{-1}
$$
where $R_0=R_0(\Omega^\prime,  \Gamma_0)>1$ and $C_2=C_2(\Omega^\prime, \Gamma_0)>0$ are the constants given by Proposition \ref{CalibrationProp}. In particular, $E_{rel}[\Gamma, \Gamma_0]$ exists (possibly infinite) and, for any $R> R_0$, satisfies the estimate
$$
E_{rel}[\Gamma, \Gamma_0]\geq E_{rel}[\Gamma, \Gamma_0; \bar{B}_R]-C_2 R^{-1}.
$$
\end{thm}

Our main tool will be the divergence theorem applied to appropriately chosen vector fields. 

\begin{lem}\label{DivergenceLem}
Suppose $\mathbf{Y}\in Lip_{loc}(\overline{\Omega^\prime}; \Real^{n+1})$ satisfies the following bounds for some constants $M_0>0$ and $\gamma_0<1$:
\begin{enumerate}
\item $\left|\Div \mathbf{Y}+\frac{\mathbf{x}}{2}\cdot \mathbf{Y}\right|\leq M_0 |\mathbf{x}|^{\gamma_0}$;
\item $|\mathbf{x}\cdot \mathbf{Y}|\leq M_0 |\mathbf{x}|^{\gamma_0+2}$.
\end{enumerate}
If $\psi_{\mathbf{Y}}\in C^0_{loc}(\overline{\Omega^\prime}\times \mathbb{S}^n)$ is defined by
$$
\psi_{\mathbf{Y}}(p, \mathbf{v})=\mathbf{Y}(p)\cdot \mathbf{v},
$$
then there is a positive constant $C_0=C_0(\Omega^\prime,\Gamma_0, \gamma_0)$ so that, for any $0<\frac{1}{2}R_1<R_1-\delta<R_1<R_2$, 
$$
\left|E[\Gamma, \Gamma_0; \alpha_{R_1,R_2,\delta}\psi_{\mathbf{Y}}]\right| \leq C_0 M_0 R_1^{\gamma_0-1}.
$$
\end{lem}

\begin{proof}
Denote by $\Omega_U^+=U\cap\Omega_+(\Gamma_0)$ and $\Omega_U^-=\left(\Real^{n+1}\backslash \overline{U}\right) \cap \Omega_-(\Gamma_0)$. The divergence theorem implies that
\begin{multline*}
\int_{\Gamma} \alpha_{R_1, R_2, \delta} \mathbf{Y} \cdot \mathbf{n}_{\Gamma} e^{\frac{|\mathbf{x}|^2}{4}} \, d\mathcal{H}^n -\int_{\Gamma_0} \alpha_{R_1, R_2, \delta} \mathbf{Y} \cdot \mathbf{n}_{\Gamma_0}e^{\frac{|\mathbf{x}|^2}{4}} \, d\mathcal{H}^n \\
= \int_{\Omega_U^+}\left(\alpha_{R_1,R_2,\delta} \left(\Div\mathbf{Y} +\frac{\mathbf{x}}{2}\cdot \mathbf{Y} \right) +\nabla \alpha_{R_1,R_2,\delta} \cdot \mathbf{Y} \right) e^{\frac{|\mathbf{x}|^2}{4}}\\
-\int_{\Omega_U^-}\left(\alpha_{R_1,R_2,\delta} \left(\Div\mathbf{Y} +\frac{\mathbf{x}}{2}\cdot \mathbf{Y} \right) +\nabla \alpha_{R_1,R_2,\delta} \cdot \mathbf{Y} \right) e^{\frac{|\mathbf{x}|^2}{4}}. 
\end{multline*}
As $\spt (\alpha_{R_1,R_2,\delta})\subseteq \bar{A}_{R_1-\delta, R_2+\delta}$ and 
$$
\nabla \alpha_{R_1,R_2,\delta}(p)=\left\{ \begin{array}{cc} \frac{\mathbf{x}(p)}{\delta |\mathbf{x}(p)|} & \mbox{if $p\in A_{R_1-\delta,R_1}$} \\  -\frac{\mathbf{x}(p)}{\delta |\mathbf{x}(p)|} & \mbox{if $p\in A_{R_2, R_2+\delta}$} \\ 0 & \mbox{otherwise} \end{array} \right.
$$ 
the hypotheses on $\mathbf{Y}$ ensure that
\begin{align*}	
&\left| \int_{\Omega_U^\pm}\left( \alpha_{R_1,R_2,\delta} \left( \Div\mathbf{Y} +\frac{\mathbf{x}}{2}\cdot \mathbf{Y} \right) +\nabla \alpha_{R_1,R_2,\delta} \cdot \mathbf{Y} \right) e^{\frac{|\mathbf{x}|^2}{4}} \right|\\
&\leq M_0 \int_{\Omega_U^\pm\cap ( \bar{A}_{R_1-\delta, R_1}\cup \bar{A}_{R_2, R_2+\delta})} \delta^{-1} |\mathbf{x}|^{\gamma_0+1}e^{\frac{|\mathbf{x}|^2}{4}}+ M_0 \int_{\Omega_U^\pm\cap \bar{A}_{R_1-\delta,R_2+\delta}}|\mathbf{x}|^{\gamma_0} e^{\frac{|\mathbf{x}|^2}{4}}\\
& \leq M_0  \int_{\Omega^\prime\cap( \bar{A}_{R_1-\delta, R_1}\cup \bar{A}_{R_2, R_2+\delta})} \delta^{-1} |\mathbf{x}|^{\gamma_0+1}e^{\frac{|\mathbf{x}|^2}{4}}+ M_0 \int_{\Omega^\prime\cap \bar{A}_{R_1-\delta,R_2+\delta}}|\mathbf{x}|^{\gamma_0} e^{\frac{|\mathbf{x}|^2}{4}}.         
\end{align*}
	
As $R_1-\delta>0$, we can use the co-area formula and Lemma \ref{SliceAreaLem} to see that
\begin{align*}
\int_{\Omega^\prime\cap \bar{A}_{R_1-\delta, R_1}} \delta^{-1} |\mathbf{x}|^{\gamma_0+1} e^{\frac{|\mathbf{x}|^2}{4}} & =\int_{R_1-\delta}^{R_1} \int_{\Omega^\prime\cap \partial B_t } \delta^{-1} t^{\gamma_0+1} e^{\frac{t^2}{4}} \, d\mathcal{H}^n dt \\
&=\int_{R_1-\delta}^{R_1} t^{\gamma_0+1} e^{\frac{t^2}{4}} \mathcal{H}^n ( \Omega^\prime\cap \partial B_t ) \, dt\\
&\leq \bar{C}^\prime_1 \delta^{-1} \int_{R_1-\delta}^{R_1} t^{\gamma_0-1} \, dt 
\end{align*}
where $\bar{C}^\prime_1$ is given by Lemma \ref{SliceAreaLem}. Hence, as $\gamma_0<1$ and $R_1-\delta>\frac{1}{2}R_1$,
$$
\int_{\Omega^\prime\cap \bar{A}_{R_1-\delta, R_1}}\delta^{-1} |\mathbf{x}|^{\gamma_0+1}e^{\frac{|\mathbf{x}|^2}{4}} \leq \bar{C}^\prime_1(R_1-\delta)^{\gamma_0-1} \leq 2^{1-\gamma_0} \bar{C}^\prime_1 R^{\gamma_0-1}_1.
$$ 
In the same way, we get
$$
\int_{\Omega^\prime\cap \bar{A}_{R_2, R_2+\delta}}\delta^{-1} |\mathbf{x}|^{\gamma_0+1}e^{\frac{|\mathbf{x}|^2}{4}} \leq \bar{C}^\prime_1 R_2^{\gamma_0-1}\leq \bar{C}^\prime_1 R_1^{\gamma_0-1}.
$$ 
	
Again, using the co-area formula and Lemma \ref{SliceAreaLem} gives that
\begin{align*}
\int_{\Omega^\prime\cap \bar{A}_{R_1-\delta,R_2+\delta}}|\mathbf{x}|^{\gamma_0} e^{\frac{|\mathbf{x}|^2}{4}}  & \leq \int_{R_1-\delta}^{R_2+\delta} t^{\gamma_0} e^{\frac{t^2}{4}} \mathcal{H}^n(\Omega^\prime\cap\partial B_t ) \, dt \\
& \leq \bar{C}^\prime_1\int_{R_1-\delta}^{R_2+\delta} t^{\gamma_0-2} \, dt\\
&\leq \frac{2^{1-\gamma_0}}{1-\gamma_0} \bar{C}^\prime_1 R_1^{\gamma_0-1}
\end{align*}
where the last inequality used that $1-\gamma_0>0$ and $R_1-\delta>\frac{1}{2} R_1$.
	
Combining the above estimates and choosing $C_0$ appropriately prove the claim.
\end{proof}

We next use a foliation near infinity by almost self-expanders to  introduce a good vector field for applying the previous lemma.  

\begin{prop}\label{FoliationProp}
There are constants $R_0=R_0(\Omega^\prime,  \Gamma_0)>1$ and $C_1=C_1(\Omega^\prime,\Gamma_0)>0$ and a smooth vector field $\mathbf{N}\colon \overline{\Omega^\prime}\backslash \bar{B}_{R_0} \to \mathbb{R}^{n+1}$ that satisfies:
\begin{enumerate}
\item \label{UnitLengthItem} $|\mathbf{N}|=1$;
\item \label{InitialItem} $\mathbf{N}|_{\Gamma_0}=\mathbf{n}_{\Gamma_0}$;
\item \label{C1EstimateItem} $|\mathbf{x}\cdot \mathbf{N}|+\sum_{i=1}^3 |\nabla^i\mathbf{N}|\leq C_1 |\mathbf{x}|^{-1}$;
\item \label{ExpanderMCItem} If $D_{\Gamma_0}$ is the signed distance to $\Gamma_0$, then 
$$
\left| \Div \mathbf{N}+\frac{\mathbf{x}}{2}\cdot\mathbf{N}+\left(|A_{\Gamma_0}|^2-\frac{1}{2}\right)D_{\Gamma_0} \right|\leq C_1 D_{\Gamma_0}^2
$$
and so 
$$
\left| \Div \mathbf{N} +\frac{\mathbf{x}}{2}\cdot \mathbf{N} \right| \leq C_1 |\mathbf{x}|^{-n-1} e^{-\frac{|\mathbf{x}|^2}{4}}.
$$
\end{enumerate}
\end{prop}

\begin{proof}
Let $\Pi_{\Gamma_0}$ be the nearest point projection to $\Gamma_0$.  As $\Gamma_0$ is $C^2$-asymptotically conical, there is an $\epsilon_0=\epsilon_0(\Gamma_0)\in (0,1)$ so that 
$$
\Psi\colon \mathcal{T}_{\epsilon_0}(\Gamma_0) \to \Gamma_0 \times (-\epsilon_0,\epsilon_0)
$$
given by $\Psi(p)=(\Pi_{\Gamma_0}(p), D_{\Gamma_0}(p))$ is a diffeomorphism. Hence, setting 
$$
\mathbf{N}(p)=\mathbf{n}_{\Gamma_0}(\Pi_{\Gamma_0}(p))
$$ 
one obtains a vector field on $\mathcal{T}_{\epsilon_0}(\Gamma_0)$ that is readily seen to satisfy Items \eqref{UnitLengthItem} and \eqref{InitialItem}.  As $\Gamma_0$ is a self-expander, both $\mathbf{n}_{\Gamma_0}$ and $\Pi_{\Gamma_0}$ are smooth and, by the chain rule, so is $\mathbf{N}$. 

By \eqref{LinearDecayEqn} with $\Gamma=\Gamma_0$,
$$
C_{\Gamma_0,3}=\sup_{q\in\Gamma_0} \left( \left(1+|\mathbf{x}(q)|\right)\sum_{i=1}^3 |\nabla^i_{\Gamma_0}\mathbf{n}_{\Gamma_0}(q)| \right) < \infty.
$$
As, up to shrinking $\epsilon_0$, one has, for $i=1,..,3$, $|\nabla^i\Pi_{\Gamma_0}(p)|\leq 2$, it follows from the chain rule that, for all $p\in\mathcal{T}_{\epsilon_0}(\Gamma_0)$,
$$
\sum_{i=1}^3 |\nabla^i\mathbf{N}(p)| \leq 2C_{\Gamma_0,3} |\Pi_{\Gamma_0}(p)|^{-1}.
$$
Observe that if $p\in\mathcal{T}_{\epsilon_0}(\Gamma_0)\setminus\bar{B}_{2\epsilon_0^{-1}}$, then 
$$
\frac{1}{2} |\mathbf{x}(p)| \leq |\Pi_{\Gamma_0}(p)| \leq 2 |\mathbf{x}(p)|
$$
and so 
$$
\sum_{i=1}^3 |\nabla^i \mathbf{N}(p)| \leq 4C_{\Gamma_0,3} |\mathbf{x}(p)|^{-1}.
$$

It is readily checked that
\begin{equation} \label{NormalPositionEqn}
\begin{split}
\mathbf{x}(p)\cdot\mathbf{N}(p) & =\left(\Pi_{\Gamma_0}(p)+D_{\Gamma_0}(p)\mathbf{n}_{\Gamma_0}(\Pi_{\Gamma_0}(p))\right)\cdot\mathbf{n}_{\Gamma_0}(\Pi_{\Gamma_0}(p)) \\
& = \Pi_{\Gamma_0}(p)\cdot\mathbf{n}_{\Gamma_0}(\Pi_{\Gamma_0}(p))+D_{\Gamma_0}(p) \\
& =-2H_{\Gamma_0}(\Pi_{\Gamma_0}(p))+D_{\Gamma_0}(p).
\end{split}
\end{equation}
As $\Omega^\prime$ is thin at infinity, the definition ensures that there is a radius $R_0=R_0(\Gamma_0,\epsilon_0, \bar{R}_0^\prime, \bar{C}_0^\prime)$ and a constant $C=C(\Gamma_0,\bar{C}_0^\prime)$ so that $\overline{\Omega^\prime}\setminus\bar{B}_{R_0}\subset\mathcal{T}_{\epsilon_0}(\Gamma_0)\setminus\bar{B}_{2\epsilon_0^{-1}}$ and, for all $p\in\Omega^\prime\setminus \bar{B}_{R_0}$,
$$
|\mathbf{x}\cdot\mathbf{N}(p)| \leq C |\mathbf{x}(p)|^{-1}.
$$
Thus we have shown Item \eqref{C1EstimateItem} as long as we choose $C_1>\max\set{4C_{\Gamma_0,3}, C}$.

To see the last claim, up to shrinking $\epsilon_0$ so $\epsilon_0<\frac{1}{8C_{\Gamma_0,3}}$ one has, for every $t\in (-\epsilon_0,\epsilon_0)$,
$$
\Upsilon_t=\set{\mathbf{x}(p)+t\mathbf{n}_{\Gamma_0}(p)\colon p\in\Gamma_0}
$$
is a hypersurface in $\mathbb{R}^{n+1}$ and, by Lemma \ref{ExpanderMeanCurvLem},
$$
\left| H_{\Upsilon_t}+\frac{\mathbf{x}}{2}\cdot\mathbf{n}_{\Upsilon_t}+\left(|A_{\Gamma_0}|^2-\frac{1}{2}\right)t  \right| \leq \bar{C}_1 t^2
$$
where $\bar{C}_1=\bar{C}_1(n,C_{\Gamma_0,3})>0$. As
$$
\Div\mathbf{N}(p)+\frac{\mathbf{x}(p)}{2}\cdot\mathbf{N}(p)=H_{\Upsilon_t}(p)+\frac{\mathbf{x}(p)}{2}\cdot\mathbf{n}_{\Upsilon_t}(p)
$$
for $p\in\Upsilon_t$ and $t=D_{\Gamma_0}(p)$, it follows that
$$
\left| \Div\mathbf{N}(p)+\frac{\mathbf{x}(p)}{2}\cdot\mathbf{N}(p)+\left(|A_{\Gamma_0}|^2-\frac{1}{2}\right)D_{\Gamma_0}(p) \right|  \leq \bar{C}_1 D_{\Gamma_0}(p)^2.
$$
The result follows by enlarging $C_1$ so that $C_1>\max\set{\bar{C}_1,\bar{C}_0^\prime(\bar{C}_1+C_{\Gamma_0,3}^2+1)}$.
\end{proof}

Using the vector field of Proposition \ref{FoliationProp} we obtain a two-sided estimate on the functional $E$ for weights near infinity.

\begin{prop}\label{CalibrationProp}
There is a constant $C_2=C_2(\Omega^\prime, \Gamma_0)>0$ so that if $\psi\in Lip(\overline{\Omega^\prime})$ satisfies $\Vert\psi\Vert_{Lip}\leq 1$ and $\psi\geq 0$, then, for any $R_0<\frac{1}{2}R_1<R_1-\delta<R_1<R_2$, 
$$
-C_2 R_1^{-1}\leq E[\Gamma, \Gamma_0; \alpha_{R_1, R_2, \delta} \psi]\leq E[\Gamma, \Gamma_0; \alpha_{R_1, R_2, \delta}]+C_2 R_1^{-1}.
$$
Here $R_0$ is the constant given by Proposition \ref{FoliationProp}.
\end{prop}

\begin{proof}
We first observe that the upper bound on $E[\Gamma, \Gamma_0; \alpha_{R_1, R_2, \delta} \psi]$ follows from the lower bound. Indeed, if $\tilde{\psi}=1-\psi$, then $\tilde{\psi}$ satisfies the same hypotheses as $\psi$ and so, assuming the lower bound holds,
$$
-C_2 R_1^{-1}\leq E[\Gamma, \Gamma_0; \alpha_{R_1, R_2, \delta} \tilde{\psi}]= E[\Gamma, \Gamma_0; \alpha_{R_1, R_2, \delta} (1-\psi)].
$$
Hence, one has that
$$
-C_2 R_1^{-1}+E[\Gamma, \Gamma_0; \alpha_{R_1, R_2, \delta} {\psi}]\leq E[\Gamma, \Gamma_0; \alpha_{R_1, R_2, \delta}],
$$
proving the upper bound.
	
In order to prove the lower bound, set $\mathbf{Y}=\psi \mathbf{N}$ where $\mathbf{N}$ is given by Proposition \ref{FoliationProp}. One computes that
$$
\Div\mathbf{Y}+\frac{\mathbf{x}}{2}\cdot \mathbf{Y}= \nabla\psi \cdot \mathbf{N}+\psi\left(\Div\mathbf{N}+\frac{\mathbf{x}}{2}\cdot \mathbf{N}\right).
$$
Thus, Proposition \ref{FoliationProp} and the assumptions on $\psi$ imply that, for $p\in\overline{\Omega^\prime}\setminus\bar{B}_{R_0}$,
$$
\left|\Div\mathbf{Y}(p)+\frac{\mathbf{x}(p)}{2}\cdot \mathbf{Y}(p)\right| \leq C_1+1.
$$
Likewise,
$$
|\mathbf{x}(p)\cdot \mathbf{Y}(p)|=\psi(p) |\mathbf{x}(p)\cdot \mathbf{N}(p)|\leq C_1 |\mathbf{x}(p)|^{-1}.
$$
Hence, as $R_0<R_1-\delta$,  appealing to Lemma \ref{DivergenceLem} gives
$$
\int_{\Gamma} \alpha_{R_1,R_2,\delta} \psi \mathbf{N}\cdot \mathbf{n}_{\Gamma} e^{\frac{|\mathbf{x}|^2}{4}} \, d\mathcal{H}^n \geq \int_{\Gamma_0} \alpha_{R_1,R_2,\delta} \psi e^{\frac{|\mathbf{x}|^2}{4}} \, d\mathcal{H}^n-C_0 (C_1+1) R_1^{-1}.
$$
However, as $\psi\geq 0$,  $\psi \mathbf{N}\cdot \mathbf{n}_{\Gamma}\leq \psi$ and so 
$$
\int_{\Gamma} \alpha_{R_1,R_2,\delta} \psi e^{\frac{|\mathbf{x}|^2}{4}} \, d\mathcal{H}^n \geq \int_{\Gamma_0} \alpha_{R_1,R_2,\delta} \psi e^{\frac{|\mathbf{x}|^2}{4}} \, d\mathcal{H}^n-C_0 (C_1+1) R_1^{-1}.
$$
That is,
$$
E[\Gamma, \Gamma_0; \alpha_{R_1,R_2,\delta} \psi ]\geq-C_2 R_1^{-1}
$$
for $C_2=C_0(C_1+1)$.
\end{proof}

We may now prove Theorem \ref{RelEntropyThm}.

\begin{proof}[Proof of Theorem \ref{RelEntropyThm}]
By the dominated convergence theorem,
$$
E_{rel}[\Gamma, \Gamma_0; \bar{B}_R]=\lim_{\delta\to 0} E[\Gamma, \Gamma_0; \phi_{R, \delta}].
$$
Proposition \ref{CalibrationProp} implies that, for any $R_2>R_1+\delta>R_1>2R_0$,
\begin{align*}
E[\Gamma, \Gamma_0; \phi_{R_2, \delta}] & =E[\Gamma, \Gamma_0; \phi_{R_1, \delta}]+E[\Gamma, \Gamma_0; \alpha_{R_1+\delta, R_2, \delta}] \\
& \geq E[\Gamma, \Gamma_0; \phi_{R_1, \delta}] -C_2 (R_1+\delta)^{-1}.
\end{align*}
The first claim follows by sending $\delta\to 0$. This implies that
$$
\liminf_{R\to \infty}E_{rel}[\Gamma, \Gamma_0; \bar{B}_R]\geq \limsup_{R\to \infty} E_{rel}[\Gamma, \Gamma_0; \bar{B}_R]
$$ 
so the limit exists. Finally, the first estimate implies the second by taking $R_2\to \infty$. 
\end{proof}

\section{Weighted relative entropy} \label{GeneralCalibrationSec}
We continue to follow the conventions of Sections \ref{Conventions} and \ref{CalibrationSec} and assume $\Gamma=\partial^* U$ for some $U\in\mathcal{C}(\Gamma_0^\prime,\Gamma_1^\prime)$. In this section we prove the generalization of Theorem \ref{WeightEstThm} to the weak setting.

\begin{thm}\label{WeightRelThm}
If $E_{rel}[\Gamma, \Gamma_0]<\infty$, then, for any $\psi \in \mathfrak{X}^e(\overline{\Omega^\prime})$, $E_{rel}[\Gamma, \Gamma_0; \psi]$ exists. Moreover, there is a constant $C_9=C_9(\Omega^\prime, \Gamma_0)>0$ so that, for all $\psi\in \mathfrak{X}^e(\overline{\Omega^\prime})$,
\begin{align*}
\left| E_{rel}[\Gamma, \Gamma_0; \psi] \right| \leq & C_9(1+ |E_{rel}[\Gamma, \Gamma_0]|)\Vert \psi \Vert_{\mathfrak{X}}.
\end{align*}	
\end{thm}

The proof of Theorem \ref{WeightRelThm} will proceed in a similar fashion to the arguments of the previous section.  In particular,  we will also use the divergence theorem, though in a more involved way.  Our first goal is to prove Theorem \ref{WeightRelThm} for weights that are of a particularly simple form -- namely modeled on a (continuously varying) quadratic form of rank at most two.  Such forms will provide good approximations to elements of $\mathfrak{X}^e$.  Here the \emph{rank} of a quadratic form $Q_A$ on $\Real^{n+1}$ is the rank of the symmetric matrix $A$ so that $Q_A(\mathbf{v})=\mathbf{v}\cdot (A\mathbf{v})$. The reason why quadratic forms of rank $2$ are relevant is that if $(\mathbf{v}, \mathbf{w})\in T_{\mathbf{v}}\mathbb{S}^n$ and $A=\mathbf{v}\mathbf{w}^\top+\mathbf{w}\mathbf{v}^\top$, then $\nabla_{\mathbb{S}^n} Q_A(\mathbf{v})=\mathbf{w}$ and $Q_A$ is the simplest even function for which this holds. 

With this in mind, for continuous vector fields $\mathbf{Y}_1,\mathbf{Y}_2$ defined on a subset $W$ of $\Real^{n+1}$, define the function $\psi_{\mathbf{Y}_1,\mathbf{Y}_2}\in C^0_{loc}(W\times\mathbb{S}^n)$ by
$$
\psi_{\mathbf{Y}_1,\mathbf{Y}_2}(p,\mathbf{v})=\psi_{\mathbf{Y}_1}(p,\mathbf{v})\psi_{\mathbf{Y}_2}(p,\mathbf{v})=(\mathbf{Y}_1(p)\cdot \mathbf{v})(\mathbf{Y}_2(p)\cdot \mathbf{v}).
$$
We first establish lower bound estimates  and a quasi-triangle inequality near infinity for rank-one quadratic forms.

\begin{lem}\label{QuadFormPrepLem}
There is a constant $C_3=C_3(\Omega^\prime, \Gamma_0)>0$ so that if $\mathbf{Y}\in Lip(\overline{\Omega^\prime}\setminus\bar{B}_{R_0};\mathbb{R}^{n+1})$ is a vector field of the form
$$
\mathbf{Y}=a\mathbf{N}+\mathbf{Z}
$$
where $|a|\leq 1$ and 
$$
\Vert |\mathbf{x}| \mathbf{Z}\Vert_{C^0}+\Vert\nabla\mathbf{Z}\Vert_{L^\infty} \leq 1,
$$
then, for any $R_0<\frac{1}{2} R_1<R_1-\delta<R_1<R_2$,
$$
E[\Gamma,\Gamma_0; \alpha_{R_1,R_2,\delta}\psi_{\mathbf{Y},\mathbf{Y}}] \geq -C_3 |E[\Gamma,\Gamma_0; \alpha_{R_1,R_2,\delta}]|-C_3 R_1^{-1}.
$$
As a consequence, if $\mathbf{Y}_i\in Lip(\overline{\Omega^\prime}\setminus\bar{B}_{R_0}; \Real^{n+1})$, $i\in\set{1, \ldots m}$, are vector fields of the form
$$
\mathbf{Y}_i=a_i \mathbf{N}+ \mathbf{Z}_i
$$
where $|a_i|\leq 1$ and
$$
\Vert |\mathbf{x}| \mathbf{Z}_i\Vert_{C^0}+\Vert\nabla\mathbf{Z}_i\Vert_{L^\infty} \leq 1,
$$
and $\mathbf{W}=\sum_{i=1}^m \mathbf{Y}_i$, then
\begin{align*}
E[\Gamma, \Gamma_0; \alpha_{R_1, R_2, \delta} \psi_{\mathbf{W}, \mathbf{W}}]& \leq 2^m \sum_{i=1}^m E[\Gamma, \Gamma_0; \alpha_{R_1, R_2, \delta}\psi_{\mathbf{Y}_i, \mathbf{Y}_i}] \\
&+ m^3 2^m C_3\left| E[\Gamma, \Gamma_0; \alpha_{R_1, R_2, \delta}]\right|+ m^3 2^m C_3 R_1^{-1}.
\end{align*}
Here $R_0$ is the constant and $\mathbf{N}$ is the vector field given by Proposition \ref{FoliationProp}.
\end{lem}

\begin{proof}
Set
$$
\bar{\mathbf{Y}}=(\mathbf{Y}\cdot \mathbf{N}) \mathbf{Y}.
$$
Using Proposition \ref{FoliationProp}, one computes that
$$
\left| \Div \bar{\mathbf{Y}}+\frac{\mathbf{x}}{2}\cdot \bar{\mathbf{Y}}\right| \leq c(n) (C_1+1)
$$
and
$$
|\mathbf{x}\cdot \bar{\mathbf{Y}}|\leq c(n) (C_1+1).
$$
Hence, by Lemma \ref{DivergenceLem},
$$
\int_{\Gamma} \alpha_{R_1,R_2,\delta} \bar{\mathbf{Y}}\cdot \mathbf{n}_{\Gamma} e^{\frac{|\mathbf{x}|^2}{4}} \, d\mathcal{H}^n \geq \int_{\Gamma_0} \alpha_{R_1,R_2,\delta}  \bar{\mathbf{Y}}\cdot \mathbf{n}_{\Gamma_0} e^{\frac{|\mathbf{x}|^2}{4}} \, d\mathcal{H}^n-c(n)C_0(C_1+1)R_1^{-1}.
$$
That is, as $\mathbf{N}|_{\Gamma_0}=\mathbf{n}_{\Gamma_0}$,
\begin{align*}
\int_{\Gamma} &\alpha_{R_1,R_2,\delta}(\mathbf{Y}\cdot \mathbf{N})(\mathbf{Y}\cdot \mathbf{n}_\Gamma) e^{\frac{|\mathbf{x}|^2}{4}} \, d\mathcal{H}^n \\ &\geq \int_{\Gamma_0} \alpha_{R_1,R_2,\delta}  \psi_{\mathbf{Y},\mathbf{Y}}(\cdot, \mathbf{n}_{\Gamma_0}(\cdot)) e^{\frac{|\mathbf{x}|^2}{4}} \, d\mathcal{H}^n  -c(n)C_0(C_1+1) R_1^{-1}.
\end{align*}
By Young's inequality, on $\Gamma$,
$$
 \frac{1}{2} \psi_{\mathbf{Y},\mathbf{Y}}(p,\mathbf{N}(p)) +\frac{1}{2} \psi_{\mathbf{Y},\mathbf{Y}}(p,\mathbf{n}_\Gamma(p))\geq (\mathbf{Y}(p)\cdot \mathbf{N}(p))(\mathbf{Y}(p)\cdot \mathbf{n}_{\Gamma}(p)),
$$
while, on $\Gamma_0$,
$$
\frac{1}{2} \psi_{\mathbf{Y},\mathbf{Y}}(p,\mathbf{N}(p))+\frac{1}{2} \psi_{\mathbf{Y},\mathbf{Y}}(p,\mathbf{n}_{\Gamma_0}(p))=\psi_{\mathbf{Y},\mathbf{Y}}(p, \mathbf{n}_{\Gamma_0}(p)).
$$
Setting $\phi_{\mathbf{Y}}(p)=\psi_{\mathbf{Y},\mathbf{Y}}(p, \mathbf{N}(p))$, this yields
$$
\frac{1}{2}E[\Gamma, \Gamma_0; \alpha_{R_1, R_2, \delta}\phi_{\mathbf{Y}}]+\frac{1}{2}E[\Gamma, \Gamma_0; \alpha_{R_1, R_2, \delta}\psi_{\mathbf{Y},\mathbf{Y}}]\geq -c(n) C_0(C_1+1)R_1^{-1}.
$$
By construction, $\phi_{\mathbf{Y}}\geq 0$ and
$$
\Vert \phi_{\mathbf{Y}} \Vert_{Lip} \leq c(n) (C_1+1).
$$
Hence, by Proposition \ref{CalibrationProp} and our previous remark, 
$$
E[\Gamma, \Gamma_0; \alpha_{R_1, R_2, \delta} \phi_{\mathbf{Y}}]\leq c(n) (C_1+1) E[\Gamma, \Gamma_0; \alpha_{R_1, R_2, \delta}]+c(n) (C_1+1)C_2 R_1^{-2}.
$$
As such,
$$
E[\Gamma, \Gamma_0; \alpha_{R_1, R_2, \delta}\psi_{\mathbf{Y},\mathbf{Y}}]\geq -C_3	|E[\Gamma, \Gamma_0; \alpha_{R_1, R_2, \delta}]|- C_3 R_1^{-1}
$$
as long as $C_3 \geq 2 c(n)(C_1+1) (C_2+1)$. This gives the desired lower bound.
	
To complete the proof set, for $1\leq k\leq m$,
$$
\mathbf{W}_k=\sum_{i=1}^k \mathbf{Y}_i=\mathbf{W}_{k-1}+\mathbf{Y}_k
$$
and, for $2\leq k \leq m$,
$$
\bar{\mathbf{W}}_k=\sum_{i=1}^{k-1} \mathbf{Y}_i-\mathbf{Y}_k=\mathbf{W}_{k-1}-\mathbf{Y}_k.
$$
Clearly, $\mathbf{W}_{m}=\mathbf{W}$ and
$$
\psi_{\mathbf{W}_{k}, \mathbf{W}_{k}}+ \psi_{\bar{\mathbf{W}}_{k}, \bar{\mathbf{W}}_{k}}=2\psi_{\mathbf{W}_{k-1}, \mathbf{W}_{k-1}}+2\psi_{\mathbf{Y}_{k}, \mathbf{Y}_{k}}.
$$
In particular, applying the lower bounds we already established to $\psi_{\bar{\mathbf{W}}_{k}, \bar{\mathbf{W}}_{k}}$, gives
\begin{multline*}
E  [\Gamma,\Gamma_0;\alpha_{R_1, R_2, \delta} \psi_{\mathbf{W}_{k}, \mathbf{W}_{k}}]\leq 2 E[\Gamma, \Gamma_0; \alpha_{R_1, R_2, \delta}\psi_{\mathbf{W}_{k-1}, \mathbf{W}_{k-1}}]\\
+2E[\Gamma, \Gamma_0; \alpha_{R_1, R_2, \delta}\psi_{\mathbf{Y}_{k}, \mathbf{Y}_{k}}] + k^2 C_3 |E[\Gamma, \Gamma_0; \alpha_{R_1, R_2, \delta}]| +k^2 C_3 R_1^{-1}.
\end{multline*}
Iterating this estimate gives
\begin{align*}
E[\Gamma, \Gamma_0;\alpha_{R_1, R_2, \delta} \psi_{\mathbf{W}, \mathbf{W}}] & \leq 2^m \sum_{k=1}^m E[\Gamma, \Gamma_0; \alpha_{R_1, R_2, \delta}\psi_{\mathbf{Y}_{k}, \mathbf{Y}_{k}}]\\
& +m^3 2^m C_3 \left| E[\Gamma, \Gamma_0; \alpha_{R_1, R_2, \delta}]\right| +m^3 2^m  C_3 R_1^{-1}.
\end{align*}
This verifies the second claim.
\end{proof}

Using a polarization identity  and the previous result, we establish a two-sided estimate near infinity for general quadratic forms of rank at most $2$.

\begin{lem}\label{LemQuadEstLem}
There is a constant $C_4=C_4(\Omega^\prime, \Gamma_0)>0$ so that if $\mathbf{Y}_1,\mathbf{Y}_2\in Lip(\overline{\Omega^\prime}\setminus\bar{B}_{R_0}; \Real^{n+1})$ are vector fields of the form
$$
\mathbf{Y}_i=a_i \mathbf{N}+ \mathbf{Z}_i
$$
where $|a_i|\leq 1$ and
$$
\Vert |\mathbf{x}| \mathbf{Z}_i\Vert_{C^0}+\Vert\nabla\mathbf{Z}_i\Vert_{L^\infty} \leq 1,
$$
then,  for any $R_0<\frac{1}{2}R_1<R_1-\delta<R_1<R_2$, 
$$
|E[\Gamma, \Gamma_0; \alpha_{R_1, R_2, \delta}\psi_{\mathbf{Y}_1, \mathbf{Y}_2}]| \leq C_4 |E[\Gamma, \Gamma_0; \alpha_{R_1, R_2, \delta}]|+C_4 R_1^{-1}.
$$
Here $R_0$ is the constant and $\mathbf{N}$ is the vector field given by Proposition \ref{FoliationProp}.
\end{lem}

\begin{proof}
We first establish the bound when
$$
\mathbf{Y}_1=\mathbf{Y}_2=\mathbf{Y}=a\mathbf{N}+\mathbf{Z}.
$$
In this case, $\psi_{\mathbf{Y}_1,\mathbf{Y}_2}=\psi_{\mathbf{Y}, \mathbf{Y}}$ and so the lower bound on $E[\Gamma, \Gamma_0; \alpha_{R_1, R_2, \delta}\psi_{\mathbf{Y}_1, \mathbf{Y}_2}]$ follows from the first part of Lemma \ref{QuadFormPrepLem} as long as $C_4\geq C_3$.
	
To prove the upper bound, we use the second part of Lemma \ref{QuadFormPrepLem} to obtain
\begin{align*}
E[\Gamma, \Gamma_0; \alpha_{R_1, R_2, \delta}\psi_{\mathbf{Y}, \mathbf{Y}}] & \leq 4 E[\Gamma, \Gamma_0; \alpha_{R_1, R_2, \delta}\psi_{a\mathbf{N}, a\mathbf{N}}] +4 E[\Gamma, \Gamma_0; \alpha_{R_1, R_2, \delta}\psi_{\mathbf{Z}, \mathbf{Z}}] \\
& + 32 C_3 \left| E[\Gamma, \Gamma_0; \alpha_{R_1, R_2, \delta}]\right|+32 C_3 R_1^{-1}.
\end{align*}
As $(\mathbf{N}\cdot \mathbf{n}_\Gamma)^2\leq 1$ on $\Gamma$ while $(\mathbf{N}\cdot \mathbf{n}_{\Gamma_0})^2=1$ on $\Gamma_0$, it follows that
$$
E[\Gamma, \Gamma_0; \alpha_{R_1, R_2, \delta}\psi_{\mathbf{N}, \mathbf{N}}]\leq E[\Gamma, \Gamma_0; \alpha_{R_1, R_2, \delta}]\leq \left| E[\Gamma, \Gamma_0; \alpha_{R_1, R_2, \delta}]\right|
$$
and so, as $|a|\leq 1$, 
$$
E[\Gamma, \Gamma_0; \alpha_{R_1, R_2, \delta}\psi_{a\mathbf{N}, a\mathbf{N}}]\leq a^2 \left| E[\Gamma, \Gamma_0; \alpha_{R_1, R_2, \delta}]\right|\leq \left| E[\Gamma, \Gamma_0; \alpha_{R_1, R_2, \delta}]\right|.
$$
	
To find an upper bound on $E[\Gamma, \Gamma_0; \alpha_{R_1, R_2, \delta}\psi_{\mathbf{Z}, \mathbf{Z}}]$, write $\mathbf{Z}=\sum_{j=1}^{n+1} z_j \mathbf{e}_j$ where $\mathbf{e}_j$ is the constant vector field given by the $j$-th coordinate vector. The estimate on $\mathbf{Z}$ implies that the $z_j$ satisfy 
$$
\Vert |\mathbf{x}| z_j\Vert_{C^0}+\Vert\nabla z_j\Vert_{L^\infty} \leq 1.
$$
By the second part of Lemma \ref{QuadFormPrepLem}, 
\begin{multline*}
E[\Gamma,\Gamma_0; \alpha_{R_1, R_2, \delta} \psi_{\mathbf{Z}, \mathbf{Z}}] \leq 2^{n+1}\sum_{j=1}^{n+1}E[\Gamma, \Gamma_0; \alpha_{R_1, R_2, \delta}\psi_{z_j\mathbf{e}_j, z_j\mathbf{e}_{j}}]\\
+ (n+1)^3 2^{n+1} C_3 \left| E[\Gamma, \Gamma_0; \alpha_{R_1, R_2, \delta}]\right|+(n+1)^3 2^{n+1} C_3 R_1^{-1}.
\end{multline*}
Observe that 
$$
\sum_{k=1}^{n+1} \psi_{z_j\mathbf{e}_k,z_j\mathbf{e}_k}(p,\mathbf{v})=z_j^2(p).
$$
Hence,
$$
\sum_{k=1}^{n+1} E[\Gamma, \Gamma_0; \alpha_{R_1, R_2, \delta}\psi_{z_j\mathbf{e}_{k}, z_j\mathbf{e}_{k}}]=E[\Gamma, \Gamma_0; \alpha_{R_1, R_2, \delta}z_j^2].
$$
By the lower bound of Lemma \ref{QuadFormPrepLem}, this implies 
$$
E[\Gamma, \Gamma_0; \alpha_{R_1, R_2, \delta}\psi_{z_j\mathbf{e}_{j}, z_j\mathbf{e}_{j}}]\leq n C_3 \left|E[\Gamma, \Gamma_0; \alpha_{R_1, R_2, \delta}]\right|+n C_3 R_2^{-1} +E[\Gamma, \Gamma_0; \alpha_{R_1, R_2, \delta} z_j^2].
$$	
Appealing to Proposition \ref{CalibrationProp}, one has
$$
E[\Gamma, \Gamma_0; \alpha_{R_1, R_2, \delta}\psi_{z_j\mathbf{e}_j, z_j\mathbf{e}_j}]\leq (n C_3+1) \left|E[\Gamma, \Gamma_0; \alpha_{R_1, R_2, \delta}]\right|+(n C_3+C_2) R_1^{-1} .
$$
Hence,
$$
E[\Gamma, \Gamma_0;\alpha_{R_1, R_2, \delta} \psi_{\mathbf{Z}, \mathbf{Z}}]\leq C_4^\prime \left| E[\Gamma, \Gamma_0; \alpha_{R_1, R_2, \delta}]\right| +C_4^\prime R_1^{-1},
$$
where $C_4^\prime$ is chosen sufficiently large depending on $C_3, C_2$ and $n$. Hence, we have proved the two-sided bound for $\psi_{\mathbf{Y}, \mathbf{Y}}$.
	
To prove the general inequality recall the polarization identity
$$
\psi_{\mathbf{Y}_1,\mathbf{Y}_2} =\frac{1}{4}\left(\psi_{\mathbf{Y}_1+\mathbf{Y}_2,\mathbf{Y}_1+\mathbf{Y}_2}- \psi_{\mathbf{Y}_1-\mathbf{Y}_2,\mathbf{Y}_1-\mathbf{Y}_2}\right).
$$
Observe that
$$
\frac{1}{4}\psi_{\mathbf{Y}_1+\mathbf{Y}_2,\mathbf{Y}_1+\mathbf{Y}_2}= \psi_{\frac{1}{2}(\mathbf{Y}_1+\mathbf{Y}_2),\frac{1}{2}(\mathbf{Y}_1+\mathbf{Y}_2)}
$$
and similarly for the second term. The vector fields $\bar{\mathbf{Y}}_1=\frac{1}{2}(\mathbf{Y}_1+\mathbf{Y}_2)$ and $\bar{\mathbf{Y}}_2=\frac{1}{2}(\mathbf{Y}_1-\mathbf{Y}_2)$ satisfy the hypotheses of the lemma and so, by what we have already shown, 
\begin{align*}
\left|E[\Gamma, \Gamma_0; \alpha_{R_1, R_2, \delta}\psi_{\mathbf{Y}_1, \mathbf{Y}_2}]\right|&\leq \left|E[\Gamma, \Gamma_0; \alpha_{R_1, R_2, \delta}\psi_{\bar{\mathbf{Y}}_1, \bar{\mathbf{Y}}_1}]\right|+\left|E[\Gamma, \Gamma_0; \alpha_{R_1, R_2, \delta}\psi_{\bar{\mathbf{Y}}_2, \bar{\mathbf{Y}}_2}]\right|\\
&\leq 2C_4^\prime \left| E[\Gamma, \Gamma_0; \alpha_{R_1, R_2, \delta}]\right| +2C_4^\prime R_1^{-1}.
\end{align*}
This verifies the lemma with $C_4=2C_4^\prime$.
\end{proof}

In order to study general functions in $\mathfrak{X}^e$ it is necessary to subtract off the appropriate quadratic approximation.  This requires suitable pointwise estimates on the approximation and its error.

\begin{lem}\label{PsiLem}
Consider the constant $R_0$ and the vector field $\mathbf{N}$ given by Proposition \ref{FoliationProp}. There is a constant $C_5=C_5(\Omega^\prime, \Gamma_0)>1$ so that if $\psi$ is an element of $\mathfrak{X}(\overline{\Omega^\prime}\setminus\bar{B}_{R_0})$ and one sets
$$
\mathbf{Z}_{\psi}(p)=\nabla_{\mathbb{S}^n}\psi(p, \mathbf{N}(p))
$$
and
$$
\bar{\psi}(p,\mathbf{v})=\psi(p, \mathbf{v})-(\mathbf{Z}_{\psi}(p)\cdot \mathbf{v})(\mathbf{N}(p)\cdot \mathbf{v}),
$$
then the following is true:
\begin{enumerate}
\item \label{PsiLemItem1} $\Vert |\mathbf{x}| \mathbf{Z}_\psi \Vert_{C^0}+\Vert \nabla\mathbf{Z}_\psi \Vert_{L^\infty} \leq C_5 \Vert \psi \Vert_{\mathfrak{X}}$;
\item \label{PsiLemItem2} $\Vert \bar{\psi} \Vert_{Lip} \leq C_5 \Vert\psi\Vert_{\mathfrak{X}}$;
\item \label{PsiLemItem3} If, in addition, $\psi$ is even, then 
$$
\left|\bar{\psi}(p, \mathbf{v})-\bar{\psi}(p, \mathbf{N}(p))\right|\leq C_5 \left( 1-(\mathbf{N}(p)\cdot \mathbf{v})^2\right)\Vert \psi \Vert_{\mathfrak{X}}.
$$	
\end{enumerate}
\end{lem}

\begin{proof}
By construction,
$$
\sup_{p\in\overline{\Omega^\prime}\setminus\bar{B}_{R_0}} |\mathbf{x}(p)| |\mathbf{Z}_\psi(p)| \leq \Vert \psi \Vert_{\mathfrak{X}}.
$$
By the chain rule and Proposition \ref{FoliationProp}, 
$$
\Vert\nabla\mathbf{Z}_\psi\Vert_{L^\infty} \leq (1+c(n)C_1)\Vert\psi\Vert_{\mathfrak{X}}.
$$
Hence, combining these estimates, Item \eqref{PsiLemItem1} follows as long as $C_5\geq 2+c(n)C_1$. And using Item \eqref{PsiLemItem1} and Proposition \ref{FoliationProp} one readily checks Item \eqref{PsiLemItem2}.
	
To see the final item observe first that if $\psi$ is even, then so is $\bar{\psi}$. In particular, it is enough to establish the estimate when $\mathbf{v}\cdot\mathbf{N}(p)\in [0, 1]$. Furthermore, if $\mathbf{v}= \mathbf{N}(p)$, then the estimate is trivial and so we may assume that $\mathbf{v}\cdot \mathbf{N}(p)\in [0, 1)$. 
	
Set 
$$
\mathbf{w}=\frac{\mathbf{v}-(\mathbf{v}\cdot \mathbf{N}(p)) \mathbf{N}(p)}{|\mathbf{v}-(\mathbf{v}\cdot \mathbf{N}(p)) \mathbf{N}(p)|}
$$
so $\mathbf{w}$ is of unit length and orthogonal to $\mathbf{N}(p)$. In particular, $\mathbf{v}=\cos \tau_0 \mathbf{N}(p)+\sin \tau_0 \mathbf{w}$ where  $\cos\tau_0 =\mathbf{N}(p)\cdot \mathbf{v}\in [0,1)$. As $\cos\tau_0\in [0,1)$,  $\tau_0 \in (0, \frac{\pi}{2}]$. It follows from the Lipschitz bound on $\nabla_{\mathbb{S}^n}\bar{\psi}(p,\cdot)$ and the fact that $\nabla_{\mathbb{S}^n} \bar{\psi}(p,\mathbf{N}(p))=0$, that, for $0\leq \tau \leq \tau_0$,
\begin{align*}
\left|\nabla_{\mathbb{S}^n} \bar{\psi}(p,\cos \tau \mathbf{N}(p)+\sin \tau \mathbf{w})\right|=\left|\int_0^\tau \frac{d}{dt} \nabla_{\mathbb{S}^n} \bar{\psi}(p,\cos t \mathbf{N}(p)+\sin t \mathbf{w})  \, dt\right| \leq c(n)\tau \Vert \bar{\psi} \Vert_{\mathfrak{X}}.
\end{align*}
Integrating this estimate yields
$$
\left|\bar{\psi}(p,\mathbf{v})-\bar{\psi}(p, \mathbf{N}(p))\right|\leq c(n)\tau_0^2 \Vert \bar{\psi} \Vert_{\mathfrak{X}}.
$$
Hence, as
$$
\tau_0^2\leq \frac{\pi^2}{4} \sin^2 \tau_0 =\frac{\pi^2}{4} (1-\cos^2 \tau_0)= \frac{\pi^2}{4} \left(1-(\mathbf{N}(p)\cdot \mathbf{v})^2\right),
$$ 
Item \eqref{PsiLemItem3} follows with $C_5\geq \frac{\pi^2}{4}c(n)$.
\end{proof}

In order to extend from the quadratic approximation to the general case we need to estimate the error and this may be thought of as a sort of bound on the weighted tilt-excess near infinity in terms of the relative entropy.

\begin{prop}\label{GradEstProp}
There is a constant $C_6=C_6(\Omega^\prime, \Gamma_0)>0$ so that, for any $R_0<\frac{1}{2}R_1<R_1-\delta<R_1<R_2$, 
$$
\int_{\Gamma}  \alpha_{R_1,R_2, \delta} \left(1-(\mathbf{N}\cdot\mathbf{n}_{\Gamma})^2\right)  e^{\frac{|\mathbf{x}|^2}{4}} \, d\mathcal{H}^n\leq 2 E[\Gamma, \Gamma_0; \alpha_{R_1, R_2, \delta}]+C_6 R_1^{-4}.
$$
Here $R_0$ is the constant and $\mathbf{N}$ is the vector field given by Proposition \ref{FoliationProp}.
\end{prop}

\begin{proof}
Applying Lemma \ref{DivergenceLem} with $\mathbf{Y}=\mathbf{N}$ and appealing to Proposition \ref{FoliationProp}, gives
$$
\int_{\Gamma} \alpha_{R_1,R_2,\delta}  \mathbf{N}\cdot \mathbf{n}_{\Gamma} e^{\frac{|\mathbf{x}|^2}{4}} \, d\mathcal{H}^n \geq \int_{\Gamma_0} \alpha_{R_1,R_2,\delta} e^{\frac{|\mathbf{x}|^2}{4}} \, d\mathcal{H}^n-C_0 C_1R_1^{-4}.
$$
Thus it follows that
$$
\int_\Gamma \alpha_{R_1,R_2,\delta} \left(1-\mathbf{N}\cdot\mathbf{n}_{\Gamma}\right) e^{\frac{|\mathbf{x}|^2}{4}} \, d\mathcal{H}^n \leq E[\Gamma,\Gamma_0; \alpha_{R_1,R_2,\delta}]+C_0 C_1 R_1^{-4}.
$$
Observe that
$$
1-(\mathbf{N}\cdot\mathbf{n}_\Gamma)^2=\left(1-\mathbf{N}\cdot\mathbf{n}_\Gamma\right) \left(1+\mathbf{N}\cdot\mathbf{n}_\Gamma\right) \leq 2 \left(1-\mathbf{N}\cdot\mathbf{n}_\Gamma\right).
$$
Hence, combining these estimates, the claim follows with $C_6=2C_0C_1$.
\end{proof}

Combining above results yields an analog of Proposition \ref{CalibrationProp} for weights in $\mathfrak{X}^e$ -- i.e., an estimate near infinity. 

\begin{prop}\label{PsiEstProp}
There is a constant $C_7=C_7(\Omega^\prime, \Gamma_0)>0$ so that if $\psi\in \mathfrak{X}^e(\overline{\Omega^\prime})$ satisfies $\Vert \psi \Vert_{\mathfrak{X}}\leq 1$ and $\psi\geq 0$, then, for any $R_0<\frac{1}{2}R_1<R_1-\delta<R_1<R_2$, 
$$
\left|E[\Gamma, \Gamma_0; \alpha_{R_1, R_2, \delta}\psi]\right|	\leq C_7\left| E[\Gamma, \Gamma_0; \alpha_{R_1, R_2, \delta}] \right|+C_7 R_1^{-1}.
$$
Here $R_0$ is the constant given by Proposition \ref{FoliationProp}.
\end{prop}

\begin{proof}
As $R_1-\delta>R_0$ and $\mathrm{spt}(\alpha_{R_1,R_2,\delta})\subseteq \bar{A}_{R_1-\delta,R_2+\delta}$, we will treat $\psi$ as an element of $\mathfrak{X}^e(\overline{\Omega^\prime}\setminus\bar{B}_{R_0})$ in the following. Set 
$$
\hat{\psi}(p,\mathbf{v})=\bar{\psi}(p,\mathbf{v})+C_5.
$$
As $\Vert \mathbf{Z}_{\psi}\Vert_{C^0} \leq C_5$ and $\psi\geq 0$, this ensures that $\hat{\psi}\geq 0$. One also has
$$
\left|\hat{\psi}(p, \mathbf{v})-\hat{\psi}(p, \mathbf{N}(p))\right|\leq C_5 \left(1-(\mathbf{N}(p)\cdot \mathbf{v})^2\right).
$$
Now let 
$$
\phi(p)=\hat{\psi}(p,\mathbf{N}(p)).
$$
Using Lemma \ref{PsiLem} and Proposition \ref{FoliationProp}, one readily checks that
$$
\Vert\phi\Vert_{Lip}\leq c(n)C_5.
$$
Hence, Proposition \ref{CalibrationProp} applied to $\phi$ gives 
$$
E[\Gamma,\Gamma_0;\alpha_{R_1,R_2,\delta}\phi]\geq -c(n) C_2 C_5 R_1^{-1}.
$$
That is, 
$$
\int_{\Gamma} \alpha_{R_1,R_2,\delta} \hat{\psi}(p, \mathbf{N}(p)) e^{\frac{|\mathbf{x}|^2}{4}} \, d\mathcal{H}^n \geq \int_{\Gamma_0} \alpha_{R_1,R_2,\delta} \hat{\psi}(p, \mathbf{n}_{\Gamma_0}(p)) e^{\frac{|\mathbf{x}|^2}{4}} \, d\mathcal{H}^n - c(n) C_2 C_5 R_1^{-1}.
$$
	
The construction of $\hat{\psi}$ ensures that
\begin{align*}
\hat{\psi}(p,\mathbf{n}_{\Gamma}(p))&=\hat{\psi}(p, \mathbf{N}(p))+\left(\hat{\psi}(p, \mathbf{n}_{\Gamma}(p))- \hat{\psi}(p, \mathbf{N}(p))\right)\\
&\geq \hat{\psi}(p, \mathbf{N}(p))-C_5\left(1-(\mathbf{N}(p)\cdot \mathbf{n}_{\Gamma}(p))^2\right).
\end{align*}
Hence,
\begin{multline*}
\int_{\Gamma} \alpha_{R_1,R_2,\delta} \left(\hat{\psi}(p, \mathbf{n}_{\Gamma}(p))+C_5 \left(1-(\mathbf{N}(p)\cdot \mathbf{n}_{\Gamma}(p))^2\right)\right)  e^{\frac{|\mathbf{x}|^2}{4}} \, d\mathcal{H}^n\\
\geq \int_{\Gamma_0} \alpha_{R_1,R_2,\delta} \hat{\psi}(p, \mathbf{n}_{\Gamma_0}(p)) e^{\frac{|\mathbf{x}|^2}{4}} \, d\mathcal{H}^n-c(n)C_2 C_5 R_1^{-1}.
\end{multline*}
Appealing to Proposition \ref{GradEstProp}, one obtains
\begin{multline*}
\int_{\Gamma} \alpha_{R_1,R_2,\delta}  \hat{\psi}(p, \mathbf{n}_{\Gamma}(p)) e^{\frac{|\mathbf{x}|^2}{4}} \, d\mathcal{H}^n-\int_{\Gamma_0} \alpha_{R_1,R_2,\delta} \hat{\psi}(p, \mathbf{n}_{\Gamma_0}(p)) e^{\frac{|\mathbf{x}|^2}{4}} \, d\mathcal{H}^n \\
\geq -2C_5\left|E[\Gamma, \Gamma_0; \alpha_{R_1, R_2, \delta}]\right|-(C_6+c(n)C_2) C_5 R_1^{-1}.
\end{multline*}
As $\hat{\psi}=\bar{\psi}+C_5$, this implies
\begin{multline*}
\int_{\Gamma} \alpha_{R_1,R_2,\delta} \bar{\psi}(p, \mathbf{n}_{\Gamma}(p)) e^{\frac{|\mathbf{x}|^2}{4}} \, d\mathcal{H}^n-\int_{\Gamma_0} \alpha_{R_1,R_2,\delta} \bar{\psi}(p, \mathbf{n}_{\Gamma_0}(p)) e^{\frac{|\mathbf{x}|^2}{4}} \, d\mathcal{H}^n \\
\geq -3 C_5\left|E[\Gamma, \Gamma_0; \alpha_{R_1, R_2, \delta}]\right|-(C_6+c(n)C_2) C_5 R_1^{-1}.
\end{multline*}
	
Hence, by Lemma \ref{LemQuadEstLem},
\begin{multline*}
\int_{\Gamma} \alpha_{R_1,R_2,\delta} {\psi}(p, \mathbf{n}_{\Gamma}(p)) e^{\frac{|\mathbf{x}|^2}{4}} \, d\mathcal{H}^n-\int_{\Gamma_0} \alpha_{R_1,R_2,\delta} {\psi}(p, \mathbf{n}_{\Gamma_0}(p)) e^{\frac{|\mathbf{x}|^2}{4}} \, d\mathcal{H}^n \\
\geq -(3C_5+C_4)\left|E[\Gamma, \Gamma_0; \alpha_{R_1, R_2, \delta}]\right|-(C_6+c(n) C_2+C_4)C_5 R_1^{-1}.
\end{multline*}
This proves the lower bound for $C_7$ sufficiently large depending on $n, C_6, C_2, C_4$ and $C_5$.
	
To prove the upper bound observe that if $\tilde{\psi}=1-\psi$, then $\tilde{\psi}$ satisfies the hypotheses of the proposition. Observe that
$$
\left| E[\Gamma, \Gamma_0; \alpha_{R_1, R_2, \delta}]\right| \geq E[\Gamma, \Gamma_0; \alpha_{R_1, R_2, \delta}]=E[\Gamma, \Gamma_0; \alpha_{R_1, R_2, \delta}(\psi+\tilde{\psi})].
$$
Hence, using the lower bound we have established, one has
\begin{align*}
\left| E[\Gamma, \Gamma_0; \alpha_{R_1, R_2, \delta}]\right| & \geq E[\Gamma, \Gamma_0; \alpha_{R_1, R_2, \delta}\psi]+ E[\Gamma, \Gamma_0; \alpha_{R_1, R_2, \delta}\tilde{\psi}]\\
& \geq E[\Gamma, \Gamma_0; \alpha_{R_1, R_2, \delta}\psi]-C_7 	\left| E[\Gamma, \Gamma_0; \alpha_{R_1, R_2, \delta}]\right|-C_7 R_1^{-1}
\end{align*}
and so the upper bound holds after, possibly, increasing $C_7$ by one.
\end{proof}

\begin{cor} \label{EstimateAtInfinityCor}	
Suppose $E_{rel}[\Gamma, \Gamma_0]<\infty$ and that $\psi \in \mathfrak{X}^e(\overline{\Omega^\prime})$ satisfies $\Vert \psi\Vert_{\mathfrak{X}} \leq 1$ and $\psi \geq 0$.  For every $\epsilon>0$, there is a radius $R_\epsilon=R_{\epsilon}(\Omega^\prime, \Gamma_0, \Gamma, \epsilon)> R_0$ so that if $R_2>R_1>R_\epsilon$, then 
$$
\left|E[\Gamma, \Gamma_0; \psi; \bar{B}_{R_2}]-E[\Gamma, \Gamma_0; \psi; \bar{B}_{R_1}]\right| \leq \epsilon.
$$
Here $R_0$ is the constant given by Proposition \ref{FoliationProp}.
\end{cor}

\begin{proof}
By the dominated convergence theorem, for any $\zeta \in \mathfrak{X}^e(\overline{\Omega^\prime})$,
$$
E_{rel}[\Gamma, \Gamma_0; \zeta; \bar{B}_{R_2}]-E_{rel}[\Gamma, \Gamma_0; \zeta; \bar{B}_{R_1}]=\lim_{\delta\to 0} E[\Gamma, \Gamma_0; \alpha_{R_2, R_1+\delta, \delta}\zeta].
$$
Hence, by Proposition \ref{PsiEstProp} and the above observation with $\zeta={\psi}$ and $\zeta=1$,  one has
\begin{multline*}
\left| E_{rel}[\Gamma, \Gamma_0; {\psi}; \bar{B}_{R_2}]-E_{rel}[\Gamma, \Gamma_0; \psi; \bar{B}_{R_1}]\right| \\
\leq C_7 \left| E_{rel}[\Gamma, \Gamma_0; \bar{B}_{R_2}]-E_{rel}[\Gamma, \Gamma_0; \bar{B}_{R_1}]\right| + C_7 R_1^{-1}.
\end{multline*}
Observe that, by Theorem \ref{RelEntropyThm} and the fact that $E_{rel}[\Gamma, \Gamma_0]<\infty$, there is an $R_{\epsilon}^\prime>0$ so that if $R>R_{\epsilon}^\prime$, then 
$$
\left|E_{rel}[\Gamma, \Gamma_0]-E_{rel}[\Gamma, \Gamma_0; \bar{B}_R]\right|\leq \frac{\epsilon}{4 C_7}.
$$
Hence, by the triangle inequality, for $R_2>R_1>R_\epsilon^\prime$, one has 
$$
\left| E_{rel}[\Gamma, \Gamma_0; \bar{B}_{R_2}]-E_{rel}[\Gamma, \Gamma_0; \bar{B}_{R_1}]\right| \leq \frac{\epsilon}{2 C_7}.
$$
Hence, setting $R_\epsilon=\max \set{ R_{\epsilon}^\prime, 2C_7 \epsilon^{-1}, R_0}$  proves the claim.
\end{proof}

\begin{prop} \label{CalibrationMainProp}
There is a constant $C_8=C_8(\Omega^\prime,\Gamma_0)>0$ so that if $\psi\in\mathfrak{X}^e(\overline{\Omega^\prime})$ satisfies  $\Vert\psi\Vert_{\mathfrak{X}} \leq 1$ and $\psi\geq 0$, then, for any $0<\delta<1$ and $R>8R_0$,
$$
\left|E[\Gamma, \Gamma_0;\phi_{R,\delta}\psi]\right|	\leq  C_8+C_8 \left| E[\Gamma, \Gamma_0; \phi_{R,\delta}]\right|.
$$
Here $R_0$ is the constant given by Proposition \ref{FoliationProp}.
\end{prop}

\begin{proof}
Set $R_1=4R_0>4$ and observe that $R>R_1>R_1-\delta>\frac{1}{2} R_1 >R_0$. One has
$$
E[\Gamma, \Gamma_0;\phi_{R,\delta}\psi]=  E[\Gamma, \Gamma_0;\phi_{R_1-\delta,\delta}\psi]+E[\Gamma, \Gamma_0;\alpha_{R_1,R,\delta}\psi].
$$
As $0\leq \psi \leq 1$, one readily sees that
$$
-\int_{\Gamma_0} \phi_{R_1-\delta,\delta} e^{\frac{|\mathbf{x}|^2}{4}} \, d\mathcal{H}^n \leq  E[\Gamma, \Gamma_0; \phi_{R_1-\delta,\delta}\psi]
$$
and
$$
E[\Gamma, \Gamma_0;\phi_{R_1-\delta,\delta}\psi]\leq E[\Gamma, \Gamma_0;\phi_{R_1-\delta,\delta}]+\int_{\Gamma_0} \phi_{R_1-\delta,\delta} e^{\frac{|\mathbf{x}|^2}{4}} \, d\mathcal{H}^n. 
$$
Hence, setting 
$$
C_8^\prime=C_8^\prime(\Gamma_0)=\int_{\Gamma_0} \phi_{R_1-\delta,\delta} e^{\frac{|\mathbf{x}|^2}{4}} \, d\mathcal{H}^n
$$
one has
$$
-C_8^\prime \leq  E[\Gamma, \Gamma_0;\phi_{R_1-\delta,\delta}\psi]\leq E[\Gamma, \Gamma_0;\phi_{R_1-\delta,\delta}]+C_8^\prime
$$
and so
$$
\left| E[\Gamma, \Gamma_0;\phi_{R_1-\delta,\delta}\psi]\right| \leq \left|E[\Gamma, \Gamma_0;\phi_{R_1-\delta,\delta}] \right| + C_8^\prime.
$$
By Proposition \ref{PsiEstProp}, 
$$
\left|E[\Gamma, \Gamma_0;\alpha_{R_1,R,\delta}\psi]\right| \leq C_7\left| E[\Gamma, \Gamma_0; \alpha_{R_1, R_2, \delta}] \right|+C_7 R_1^{-1}.
$$
Finally, Proposition \ref{CalibrationProp} implies that
$$
E[\Gamma, \Gamma_0;\phi_{R_1-\delta,\delta}]-C_2 R_1^{-1} \leq E[\Gamma, \Gamma_0;\phi_{R_1-\delta,\delta}]+E[\Gamma, \Gamma_0;\alpha_{R_1,R,\delta}]=E[\Gamma, \Gamma_0;\phi_{R,\delta}]
$$
and so
$$
\left|E[\Gamma, \Gamma_0;\phi_{R_1-\delta,\delta}]\right|\leq C_8^\prime+C_2 R_1^{-1}+\left|E[\Gamma, \Gamma_0;\phi_{R,\delta}]\right|.
$$
Likewise,
$$
E[\Gamma, \Gamma_0;\alpha_{R_1,R,\delta}]-C^\prime_8\leq E[\Gamma, \Gamma_0;\phi_{R_1-\delta,\delta}]+E[\Gamma, \Gamma_0;\alpha_{R_1,R,\delta}]=E[\Gamma, \Gamma_0;\phi_{R,\delta}]
$$
and so
$$
\left|E[\Gamma, \Gamma_0;\alpha_{R_1, R,\delta}]\right|\leq C_2 R_1^{-1}+C_8^\prime +\left|E[\Gamma, \Gamma_0;\phi_{R,\delta}]\right|.
$$
Hence, 
$$
\left|E[\Gamma, \Gamma_0;\phi_{R,\delta}\psi]\right|\leq C_8^\prime+C_7 R_1^{-1}+(C_7+1)(C_8^\prime+C_2R^{-1}_1) +(1+C_7)\left|E[\Gamma, \Gamma_0;\phi_{R,\delta}]\right|.
$$
and the claim follows by choosing $C_8$ large enough.
\end{proof}

We now prove Theorem \ref{WeightRelThm}.

\begin{proof}[Proof of Theorem \ref{WeightRelThm}]
If $\Vert\psi\Vert_{\mathfrak{X}}=0$, then the theorem holds trivially. So suppose $\Vert\psi\Vert_{\mathfrak{X}}\neq 0$ and set $\hat{\psi}=\frac{1}{2\Vert \psi\Vert_{\mathfrak{X}}} \left(\psi +\Vert \psi\Vert_{\mathfrak{X}}\right)$. Observe that $\hat{\psi}\geq 0$ and $\Vert \hat{\psi}\Vert_{\mathfrak{X}}\leq 1$. As $E_{rel}[\Gamma, \Gamma_0]<\infty$, it is  an immediate consequence of Corollary \ref{EstimateAtInfinityCor} that 
$$
E_{rel}[\Gamma, \Gamma_0, \hat{\psi}]=\lim_{R\to \infty} E_{rel}[\Gamma, \Gamma_0; \hat{\psi}; \bar{B}_R]
$$
exists and is finite.

By the dominated convergence theorem,
$$
E_{rel}[\Gamma, \Gamma_0; \hat{\psi}; \bar{B}_R]=\lim_{\delta\to 0} E[\Gamma, \Gamma_0;  \phi_{R, \delta}\hat{\psi}].
$$
Hence, for $R>4R_0$, it follows from Proposition \ref{CalibrationMainProp} by taking $\delta\to 0$ that
$$
| E_{rel}[\Gamma, \Gamma_0; \hat{\psi}; \bar{B}_R]| \leq C_8+ C_8 |E_{rel}[\Gamma, \Gamma_0; \bar{B}_R]|.
$$
Taking the limit as $R\to \infty$, which is well defined on both sides by Theorem \ref{RelEntropyThm} and what we have already shown, gives 
$$
| E_{rel}[\Gamma, \Gamma_0; \hat{\psi}]| \leq C_8+ C_8 |E_{rel}[\Gamma, \Gamma_0]|.
$$
Finally, by linearity of $\zeta\mapsto E_{rel}[\Gamma, \Gamma_0; \zeta]$ and the triangle inequality one has
$$
|E_{rel}[\Gamma, \Gamma_0; {\psi}]| \leq 2 C_8 \left(1+2  |E_{rel}[\Gamma, \Gamma_0]|\right)\Vert \psi \Vert_{\mathfrak{X}}
$$
and so the claim follows by setting $C_9=4 C_8$.
\end{proof}  

Finally, we record the following analog of the dominated convergence theorem for the $E_{rel}$ functional. 

\begin{prop}\label{DominatedConvProp}
Suppose $E_{rel}[\Gamma, \Gamma_0]<\infty$. If $\psi_i\in \mathfrak{X}^e(\overline{\Omega^\prime})$ is a sequence with $\Vert \psi_i \Vert_{\mathfrak{X}} \leq M_1<\infty$ and so that $\psi_i\to \psi_\infty$ pointwise, where $\psi_\infty\in \mathfrak{X}^e(\overline{\Omega^\prime})$ satisfies $\Vert \psi_\infty\Vert_{\mathfrak{X}}\leq M_1$, then
$$
\lim_{i\to \infty} E_{rel}[\Gamma, \Gamma_0; \psi_i]=E_{rel}[\Gamma, \Gamma_0; \psi_\infty].
$$
\end{prop}

\begin{proof}
For $1\leq i \leq \infty$, set $\hat{\psi}_i=\frac{1}{2M_1}\left(\psi_i+M_1\right)$ and observe that $\Vert \hat{\psi}_i \Vert_{\mathfrak{X}}\leq 1$ and $\hat{\psi}_i \geq 0$. For every $\epsilon>0$, Corollary \ref{EstimateAtInfinityCor} implies that there is an $R_{\epsilon}>R_0$ so that, for all $R>R_{\epsilon}$ and all $1\leq i \leq \infty$, 
$$
 \left| E_{rel} [\Gamma, \Gamma_0; \hat{\psi}_i ]- E_{rel} [\Gamma, \Gamma_0; \hat{\psi}_i; \bar{B}_R]\right|< \frac{\epsilon}{3}.
$$
By the dominated convergence theorem, 
$$
\lim_{i\to \infty} 	E_{rel}[\Gamma, \Gamma_0;  \hat{\psi}_i; \bar{B}_{2R_{\epsilon}}]=	E_{rel}[\Gamma, \Gamma_0;  \hat{\psi}_\infty; \bar{B}_{2R_{\epsilon}}].
$$
Hence, there is an $i_0$ so that for $i\geq i_0$ one has
$$
 \left| E_{rel}[\Gamma, \Gamma_0;  \hat{\psi}_i; \bar{B}_{2R_{\epsilon}}]-E_{rel}[\Gamma, \Gamma_0;  \hat{\psi}_\infty; \bar{B}_{2R_{\epsilon}}]\right| <\frac{\epsilon}{3}.
$$
It follows from the triangle inequality that, for $i\geq i_0$, 
$$
\left| E_{rel} [\Gamma, \Gamma_0; \hat{\psi}_i ]- E_{rel} [\Gamma, \Gamma_0; \hat{\psi}_\infty]\right|< \epsilon.
$$
That is, 
$$ \lim_{i\to \infty} E_{rel} [\Gamma, \Gamma_0; \hat{\psi}_i ]= E_{rel} [\Gamma, \Gamma_0; \hat{\psi}_\infty ].
$$
The result then follows by the linearity of $\zeta \mapsto E_{rel}[\Gamma, \Gamma_0; \zeta]$.
\end{proof}

\section{$E_{rel}$-minimizers} \label{MinimizationSec}
Continue to use the conventions of Section \ref{Conventions}. In this section we use the previously established facts about $E_{rel}[\cdot, \Gamma_0]$ to show that this functional is coercive and lower-semi-continuous in an appropriate sense.  Hence,  there is a minimizer of $E_{rel}$ in $\mathcal{C}(\Gamma_0, \Gamma_1)$.  As this minimizer is a local $E$-minimizer, when $2\leq n\leq 6$, Theorem \ref{MinThm} follows immediately from this by standard regularity results.  

\begin{thm}\label{LowRegExistThm}
There is a Caccioppoli set $U_{min}\in \mathcal{C}(\Gamma_0, \Gamma_1)$ with $\Gamma_{min}=\partial^* U_{min}$ a critical point of the functional $E$ so that, for all $U\in  \mathcal{C}(\Gamma_0, \Gamma_1)$, 
$$
E_{rel}[\partial^* U, \Gamma_0]\geq E_{rel}[\partial^* U_{min}, \Gamma_0].
$$ 
Moreover, if $2\leq n \leq 6$, then $\Gamma_{min}$ is a smooth self-expander.
\end{thm}

\begin{proof}
Set $E_{min}=\inf\set{ E_{rel}[\partial^* U, \Gamma_0] \colon U\in  \mathcal{C}(\Gamma_0, \Gamma_1)}$. By Theorem \ref{RelEntropyThm}, there is a constant $\bar{E}=\bar{E}(\Gamma_1, \Gamma_0)\geq 0$ so that, for all $U\in \mathcal{C}(\Gamma_0,  \Gamma_1)$, 
$$
E_{rel}[\Gamma, \Gamma_0]\geq -\bar{E}.
$$
Hence, if $U_i$ is a minimizing sequence in $\mathcal{C}(\Gamma_0, \Gamma_1)$ for $E_{rel}[\cdot, \Gamma_0]$, then
$$
\lim_{i\to \infty} E_{rel}[\partial^* U_i,  \Gamma_0]=E_{min}\geq -\bar{E}>-\infty.
$$
and so, up to throwing out finitely many terms, one has
$$
 E_{min}\leq E_{rel}[\partial^* U_i, \Gamma_0]\leq E_{min}+1.
$$

For $R>0$, 
$$
P_{\bar{B}_R}(U_i)\leq \int_{\bar{B}_R\cap \partial^* U_i} e^{\frac{|\mathbf{x}|^2}{4}} \leq  E_0(R)+E_{rel}[\partial^* U_i, \Gamma_0; \bar{B}_R].
$$
Here $P_{\bar{B}_R}(U_i)$ is the perimeter of $U_i$ inside $\bar{B}_R$ and $E_0(R)=\int_{\bar{B}_R\cap \Gamma_0} e^{\frac{|\mathbf{x}|^2}{4}} \, d\mathcal{H}^n$. It follows from Theorem \ref{RelEntropyThm} that, for any $R>R_0$, 
$$
E_{rel}[\partial^* U_i, \Gamma_0; \bar{B}_R] \leq E_{rel}[\partial^* U_i, \Gamma_0] +C_2 R^{-1}
$$
and so, for any $R>R_0$ fixed, 
$$
P_{\bar{B}_R}(U_i)\leq M=E_0(R)+E_{min}+1+C_2R^{-1}<\infty
$$
is uniformly bounded independent of $i$.

Hence, by the standard compactness theorem for Caccioppoli sets, up to passing to a subsequence and relabeling,  $U_i\to U_\infty$ where $U_{\infty}$ is a Caccioppoli set in $\mathcal{C}(\Gamma_0 ,\Gamma_1)$ and the convergence is in the topology of Caccioppoli sets (i.e., $\mathbf{1}_{U_i}\to \mathbf{1}_{U_{\infty}}$ in the weak-* topology of  $BV_{loc}$). It follows from Theorem \ref{RelEntropyThm} that, for all $R>R_0$, 
$$
E_{rel}[\partial^* U_i; \Gamma_0]\geq E_{rel}[\partial^* U_i, \Gamma_0; \bar{B}_R] -C_2 R^{-1}
$$
Hence, passing to a limit and using the nature of the convergence of $U_i\to U_{\infty}$,
\begin{align*}
E_{min} &=\lim_{i\to \infty} E_{rel}[\partial^* U_i, \Gamma_0]\geq \liminf_{i\to \infty} \left(E_{rel}[\partial^* U_i,  \Gamma_0; \bar{B}_R] -C_2 R^{-1}\right)\\
&\geq  E_{rel}[\partial^* U_\infty,  \Gamma_0; \bar{B}_R] -C_2 R^{-1}.
\end{align*}
Taking $R\to \infty$ and appealing to Theorem \ref{RelEntropyThm} gives $E_{min} \geq E_{rel}[\partial^* U_\infty,  \Gamma_0]$.  As $E_{min}$ is the infimum of $E_{rel}[\cdot,\Gamma_0]$ in $\mathcal{C}(\Gamma_0, \Gamma_1) $ and $U_\infty \in\mathcal{C}(\Gamma_0, \Gamma_1) $, $E_{min}= E_{rel}[\partial^* U_\infty; \Gamma_0]$ and so the infimum is achieved. Hence, it remains only to show that $\Gamma_{min}=\partial^* U_{\infty}$ is a self-expander. However, it is clear that $\partial^* U_{\infty}$ must be (locally) $E$-minimizing in $\mathrm{cl}(U_1)\backslash U_0$ as otherwise $E_{min}$ would not be the infimum of $E_{rel}[\cdot,  \Gamma_0]$.

When $2\leq n \leq 6$, standard regularity theory for minimizing sets with obstacles, e.g., \cite[Section 37]{Simon}, implies $\Gamma_{min}$ is a smooth self-expander each of whose components is either entirely disjoint from $\Gamma_0\cup \Gamma_1$ or entirely agrees with a component of $\Gamma_0\cup \Gamma_1$. That is, $\Gamma_{min}\in\mathcal{H}(\Gamma_0,\Gamma_1)$.
\end{proof}

By adapting the approach sketched by Ilmanen \cite{IlmanenLec} and carried out by Ding \cite{Ding} to the obstacle setting, one may use standard GMT methods to construct a local $E$-minimizer in $\mathcal{H}(\Gamma_0, \Gamma_1)$.  Combined with Remark \ref{MinimizerRemark}, this gives an alternative approach to Theorem \ref{LowRegExistThm}.

\section{Forward monotonicity} \label{MonotoneSec}
We continue to follow the conventions of Section \ref{Conventions} and Section \ref{CalibrationSec}. Following Ilmanen \cite[Section 6]{IlmanenElliptic} (cf. \cite{Brakke}), a \emph{Brakke flow} is a family of Radon measures $\set{\mu_t}_{t\in (0,T)}$ on $\mathbb{R}^{n+1}$ which satisfies, for all non-negative $\psi\in C^1_{c}(\mathbb{R}^{n+1})$ and all $0<t_0\leq t_1 <T$,
$$
\int \psi \, d\mu_{t_1} \leq \int \psi \, d\mu_{t_0}+ \int_{t_0}^{t_1}\int \left(-\psi |\mathbf{H}|^2+\nabla\psi\cdot S^\perp\cdot\mathbf{H}\right) d\mu_t dt.
$$
Here $S=S(\mathbf{x})=T_{\mathbf{x}}\mu_t$ is the generalized tangent plane of $\mu_t$ at $\mathbf{x}$ and $\mathbf{H}=\mathbf{H}_{\mu_t}$ is the generalized mean curvature vector of $\mu_t$. The inner integral on the right-hand side of the inequality is interpreted according to the convention that if any quantities are not defined, then take the integral to be $-\infty$. We call a Brakke flow $\set{\mu_t}_{t\in (0,T)}$ \emph{integral} if $\mu_t$ has integer multiplicity for a.e. $t$. It is technically convenient to restrict our study to the smaller class of integral Brakke flows defined in \cite[Section 7]{WhiteLocReg}. This class is compact under the convergence of Brakke flows and is quite general, for instance it includes the flows constructed by Ilmanen's elliptic regularization procedure.

In this section we prove a version of weighted forward monotonicity formula and use it to show the asymptotic behavior of flows coming out of a cone. Theorem \ref{MonotoneThm} is a special case of the following theorem.

\begin{thm} \label{WeightMonotoneThm}
Let $\set{\mu_t}_{t\in (0,T)}$ be an integral Brakke flow that satisfies
\begin{enumerate}
\item $\lim_{t\to 0} \mu_t=\mathcal{H}^n\lfloor\mathcal{C}$;
\item For each $t\in (0,T)$, $t^{-1/2}\mathrm{spt}(\mu_t) \subseteq \overline{\Omega^\prime}$.
\end{enumerate}
For any sequence $t_i\to 0$, there is a subsequence $t_{i_j}\to 0$ and a (possibly singular) self-expander $\hat{\nu}$ asymptotic to $\mathcal{C}$ and with $\spt(\hat{\nu})\subseteq\overline{\Omega^\prime}$ so that 
$$
\mathscr{D}_{t^{-1/2}_{i_j}}\mu_{t_{i_j}} \to \hat{\nu}.
$$
Here, for a measure $\mu$ and $\rho>0$, $\mathscr{D}_\rho\mu$ is the measure given by 
$$
\mathscr{D}_\rho\mu(Y)=\rho^n\mu(\rho^{-1}Y) \mbox{ for all $\mu$-measurable subsets $Y\subseteq\mathbb{R}^{n+1}$}.
$$
\end{thm}

In order to prove Theorem \ref{WeightMonotoneThm}, we will need several auxiliary lemmas and propositions. The first two of these show the relative entropy near infinity is arbitrarily small for $C^2$-asymptotically conical ends trapped between the ends of $\Gamma_0^\prime$ and $\Gamma_1^\prime$.  The computations are very similar in spirit to those of \cite[Proposition 3.1]{DeruelleSchulze}.

\begin{lem} \label{UniformGraphLem}
Fix $\hat{C}_0>0$ and $\hat{R}_0>1$. There is a radius $\hat{R}_1=\hat{R}_1(\Gamma_0,\Omega^\prime,\hat{C}_0,\hat{R}_0)>\hat{R}_0$ so that if $\Gamma\in\mathcal{H}(\Gamma_0^\prime\setminus\bar{B}_{\hat{R}_0},\Gamma_1^\prime\setminus\bar{B}_{\hat{R}_0})$ is asymptotic to $\cC$ and satisfies  
$$
\sup_{p\in\Gamma} |\mathbf{x}(p)| |A_\Gamma(p)| \leq \hat{C}_0,
$$
then there is a smooth function $v\colon\Gamma_0\setminus\bar{B}_{\hat{R}_1}\to \mathbb{R}$ with $\Vert\nabla_{\Gamma_0} v \Vert_{C^0} \leq 1$ so that 
$$
\Gamma\setminus\bar{B}_{2\hat{R}_1}\subset\set{\mathbf{x}(p)+v(p)\mathbf{n}_{\Gamma_0}(p)\colon p\in\Gamma_0\setminus\bar{B}_{\hat{R}_1}}\subset\Gamma.
$$
\end{lem}

\begin{proof} 
Our hypotheses on $\Gamma$ ensures that it is embedded and $C^1$-asymptotic to $\cC$. Thus it is enough to prove that there is a uniform radius outside of which $\Gamma$ is a local graph over $\Gamma_0$ with the desired estimates. This is proved by contradiction. Indeed, suppose there was no such radius, then there would be a sequence of hypersurfaces $\Upsilon_i$ in $\mathbb{R}^{n+1}\setminus\bar{B}_{\hat{R}_0}$ satisfying the hypotheses and a sequence of points $q_i\in\Upsilon_i\cap\partial B_{R_i}$ with $R_i\geq\hat{R}_0$ going to infinity so that if $p_i$ is the nearest point projection of $q_i$ to $\Gamma_0$, then $|\mathbf{n}_{\Upsilon_i}(q_i)\cdot\mathbf{n}_{\Gamma_0}(p_i)|<\epsilon$ for some fixed $\epsilon\in (0,1)$. Up to passing to a subsequence and relabeling, $R_i^{-1} q_i\to q$ for some $q\in\mathcal{C}\cap \partial B_1$. Thus, by the linear decay on $|A_{\Upsilon_i}|$, it follows from the Arzel\`{a}-Ascoli theorem that, up to passing to a subsequence and relabeling, the $R_i^{-1}\Upsilon_i\cap B_1(R_i^{-1} q_i)$ converges in the $C^1$ topology to a $C^2$- hypersurface, $\Sigma$, in $B_1(q)$ which transversally intersects $\mathcal{C}$ at $q$. However, as $\Gamma_0$, $\Gamma_0'$ and $\Gamma_1'$ are all asymptotic to $\mathcal{C}$, the hypotheses the $\Upsilon_i$ satisfy imply that $\Sigma$ must be contained in $\mathcal{C}$. This is a contradiction. 
\end{proof}

\begin{prop} \label{RelEntropyAnnuliProp}
Fix $\hat{C}_0>0$ and $\hat{R}_0>1$. There is a radius $\hat{R}_2=\hat{R}_2(\Gamma_0,\Omega^\prime,\hat{C}_0,\hat{R}_0)>\hat{R}_0$ and a constant $\hat{C}_1=\hat{C}_1(\Gamma_0, \Omega^\prime,\hat{C}_0)>0$ so that if $\Gamma\in\mathcal{H}(\Gamma_0'\setminus\bar{B}_{\hat{R}_0},\Gamma_1'\setminus\bar{B}_{\hat{R}_0})$ is asymptotic to $\cC$ and satisfies 
$$
\sup_{p\in\Gamma} |\mathbf{x}(p)| |A_\Gamma(p)| \leq \hat{C}_0,
$$
then, for any $R_2>R_1>\hat{R}_2$ and $0<\delta<1$,
$$
\left| E[\Gamma,\Gamma_0; \alpha_{R_1,R_2,\delta}] \right| \leq \hat{C}_1 R_1^{-2}.
$$
\end{prop}

\begin{proof}
By the definition of thin at infinity relative to $\Gamma_0$ and Lemma \ref{UniformGraphLem}, there is a radius $\hat{R}_2^\prime>\max\set{\bar{R}_0^\prime,\hat{R}_1}$, depending on $\Gamma_0,\Omega^\prime,\hat{C}_0$ and $\hat{R}_0$, so that there is a smooth function $v\colon\Gamma_0\setminus\bar{B}_{\hat{R}_2^\prime}\to\mathbb{R}$ which satisfies
$$
\Gamma\setminus\bar{B}_{2\hat{R}_2^\prime}\subset\set{\mathbf{x}(p)+v(p)\mathbf{n}_{\Gamma_0}(p)\colon p\in\Gamma_0\setminus\bar{B}_{\hat{R}_2^\prime}}\subset\Gamma
$$
and 
$$
|v(p)| \leq 2\bar{C}_0^\prime |\mathbf{x}(p)|^{-n-1} e^{-\frac{|\mathbf{x}(p)|^2}{4}} <1.
$$
Here $\bar{R}_0^\prime=\bar{R}_0^\prime(\Gamma_0, \Omega^\prime)$ and $\bar{C}_0^\prime=\bar{C}_0^\prime(\Gamma_0, \Omega^\prime)$ are determined from the definition of thin at infinity. By the linear decay of $|A_\Gamma|$ and the gradient estimate from Lemma \ref{UniformGraphLem}, there is a constant $K_0=K_0(\Gamma_0,\hat{C}_0)>0$ so that 
$$
|\nabla_{\Gamma_0}^2 v (p)| \leq K_0 |\mathbf{x}(p)|^{-1}.
$$
Thus, by the interpolation inequality \cite[Lemma 6.32]{GilbargTrudinger}, there is a $K_1=K_1(\Gamma_0,\bar{C}_0^\prime,K_0)$ (which, in turn, depends on $\Gamma_0,\Omega^\prime$ and $\hat{C}_0$) so that 
$$
|\nabla_{\Gamma_0} v(p)|^2\leq K_1 |\mathbf{x}(p)|^{-n-2} e^{-\frac{|\mathbf{x}(p)|^2}{4}}.
$$

For $0\leq s\leq 1$, let 
$$
\hat{\Gamma}_s=\set{\mathbf{f}_s(p)=\mathbf{x}(p)+sv(p)\mathbf{n}_{\Gamma_0}(p)\colon p\in\Gamma_0\setminus\bar{B}_{\hat{R}_2^\prime}}. 
$$
Observe that $\hat{\Gamma}_0=\Gamma_0\setminus\bar{B}_{\hat{R}^\prime_2}$ and $\Gamma\setminus\bar{B}_{2\hat{R}_2^\prime}\subset\hat{\Gamma}_1\subset\Gamma$. If $R_2>R_1>4\hat{R}_2^\prime$, then the bound on $v$ ensures $\spt(\alpha_{R_1,R_2,\delta})\subset\hat{\Gamma}_s$. Thus, by the first variation formula,
\begin{align*}
\frac{d}{ds} \int_{\hat{\Gamma}_s} \alpha_{R_1,R_2,\delta} e^{\frac{|\mathbf{x}|^2}{4}} \, d\mathcal{H}^n
& =\int_{\hat{\Gamma}_s} -\alpha_{R_1,R_2,\delta} \mathbf{Y}_s\cdot\left(\mathbf{H}_{\hat{\Gamma}_s}-\frac{\mathbf{x}^\perp}{2}\right) e^{\frac{|\mathbf{x}|^2}{4}} \, d\mathcal{H}^n \\
&+\int_{\hat{\Gamma}_s} \nabla\alpha_{R_1,R_2,\delta}\cdot \mathbf{Y}_s^\perp e^{\frac{|\mathbf{x}|^2}{4}} \, d\mathcal{H}^n \\
& =: I+II
\end{align*}
where $\mathbf{Y}_s=(v\mathbf{n}_{\Gamma_0})\circ\mathbf{f}_s^{-1}$ is a vector field along $\hat{\Gamma}_s$. By the above established estimates for $v$ and $\nabla_{\Gamma_0} v$ and enlarging $\hat{R}_2^\prime$ if needed, it is readily checked that, for any $0\leq s\leq 1$ and $p\in\Gamma_0\setminus\bar{B}_{\hat{R}_2^\prime}$,
$$
e^{\frac{|\mathbf{f}_s(p)|^2}{4}} \, dvol_{\hat{\Gamma}_s}(\mathbf{f}_s(p)) \leq 2 e^{\frac{|\mathbf{x}(p)|^2}{4}} \, dvol_{\Gamma_0}(p) \mbox{ and } |\nabla_{\Gamma_0}|\mathbf{f}_s(p)||\geq \frac{1}{2},
$$
and there is a $K_2=K_2(\Gamma_0,\bar{C}_0^\prime, K_1)>0$, thus depending on $\Gamma_0,\Omega^\prime,\hat{C}_0$, so that, for all $R>\hat{R}_2^\prime$,
$$
\mathcal{H}^{n-1}(\set{|\mathbf{f}_s|=R}) \leq K_2 R^{n-1}.
$$
One also appeals to the estimates for $v$ and $\nabla_{\Gamma_0}^i v$ and Lemma \ref{ExpanderMeanCurvLem} to see that if $s\in [0,1]$ and $p\in\Gamma_0\setminus\bar{B}_{\hat{R}_2^\prime}$, then
$$
\left| \mathbf{H}_{\hat{\Gamma}_s}-\frac{\mathbf{x}^\perp}{2} \right| (\mathbf{f}_s(p)) \leq K_3 |\mathbf{x}(p)|^{-1}
$$
where $K_3=K_3(\Gamma_0,\Omega^\prime,\hat{C}_0)>0$. Thus, using these estimates and the co-area formula one computes that
\begin{align*}
\left| I \right| & \leq  2K_3\int_{\Gamma_0} (\alpha_{R_1,R_2,\delta}\circ\mathbf{f}_s) |v| |\mathbf{x}|^{-1} e^{\frac{|\mathbf{x}|^2}{4}} \, d\mathcal{H}^n \\
& \leq 2K_3\int_{R_1-2}^{R_2+2} \int_{\Gamma_0\cap\partial B_t} |v| t^{-1} e^{\frac{t^2}{4}} \frac{1}{|\nabla_{\Gamma_0} |\mathbf{x}||}\, d\mathcal{H}^{n-1} dt \\
& \leq 8\bar{C}_0^\prime K_3 \int_{R_1-2}^{R_2+2} t^{-n-2} \mathcal{H}^{n-1}(\Gamma_0\cap\partial B_t) \, dt \\
& \leq 4\bar{C}_0^\prime K_3 K_2 (R_1-2)^{-2}
\end{align*}
where the second inequality used that $\spt(\alpha_{R_1,R_2,\delta}\circ\mathbf{f}_s)\subseteq \bar{A}_{R_1-2,R_2+2}$ as $\spt(\alpha_{R_1,R_2,\delta})\subseteq\bar{A}_{R_1-\delta,R_2+\delta}$ and 
$$
|\mathbf{f}_s(p)-\mathbf{x}(p)| < 1.
$$
Likewise, one has
\begin{align*}
\left| II \right| & \leq 2 \int_{\Gamma_0} |\nabla\alpha_{R_1,R_2,\delta}\circ\mathbf{f}_s| |v| e^{\frac{|\mathbf{x}|^2}{4}} \, d\mathcal{H}^n \\
& \leq 2\delta^{-1} \int_Y \int_{\set{|\mathbf{f}_s|=t}} |v| e^{\frac{|\mathbf{x}|^2}{4}} \frac{1}{|\nabla_{\Gamma_0}|\mathbf{f}_s||} \, d\mathcal{H}^{n-1}dt \\
& \leq 8\delta^{-1} \bar{C}_0^\prime \int_Y (t-1)^{-n-1} \mathcal{H}^{n-1}(\set{|\mathbf{f}_s|=t}) \, dt\\
&\leq 64\bar{C}_0^\prime K_2 (R_1-2)^{-2}
\end{align*}
where the second inequality used that $\spt(\nabla\alpha_{R_1,R_2,\delta})\subseteq\bar{A}_{R_1-\delta,R_1}\cup\bar{A}_{R_2,R_2+\delta}$ and $Y=[R_1-\delta,R_1]\cup [R_2,R_2+\delta]$. Hence, combining estimates on $I$ and $II$ gives that, as $R_1-2>\frac{1}{2} R_1$,
$$
\left| \frac{d}{ds} \int_{\hat{\Gamma}_s} \alpha_{R_1,R_2,\delta} e^{\frac{|\mathbf{x}|^2}{4}} \, d\mathcal{H}^n \right| \leq \hat{C}_1 R_1^{-2}
$$
where $\hat{C}_1=2^{8}\bar{C}_0^{\prime} K_2(K_3+1)$ depends on $\Gamma_0,\Omega^\prime$ and $\hat{C}_0$. Therefore, 
$$
\left| E[\Gamma,\Gamma_0; \alpha_{R_1,R_2,\delta}] \right| \leq \int_0^1 \left| \frac{d}{ds} \int_{\hat{\Gamma}_s} \alpha_{R_1,R_2,\delta} e^{\frac{|\mathbf{x}|^2}{4}} \, d\mathcal{H}^n \right| ds  \leq \hat{C}_1 R_1^{-2}
$$
and so the claim follows with $\hat{R}_2=4\hat{R}_2^\prime$.
\end{proof}

Given a Brakke flow $\set{\mu_t}_{t\in (0,T)}$ set
$$
\nu_s=\mathscr{D}_{t^{-1/2}}\mu_t \mbox{ where $s=\log t$}.
$$
One readily verifies that $\set{\nu_s}_{s<\log T}$ satisfies, for all nonnegative $\psi\in C_c^1(\mathbb{R}^{n+1})$ and all $-\infty<s_0\leq s_1<\log T$,
\begin{align*}
\int\psi e^{\frac{|\mathbf{x}|^2}{4}}\, d\nu_{s_1} & \leq \int\psi e^{\frac{|\mathbf{x}|^2}{4}}\, d\nu_{s_0}-\int_{s_0}^{s_1} \int \psi \left|\mathbf{H}-\frac{\mathbf{x}}{2}\cdot S^\perp\right|^2 e^{\frac{|\mathbf{x}|^2}{4}} \, d\nu_s ds \\
& +\int_{s_0}^{s_1} \int \nabla\psi\cdot S^\perp \cdot \left(\mathbf{H}-\frac{\mathbf{x}}{2}\right) e^{\frac{|\mathbf{x}|^2}{4}} \, d\nu_s ds.
\end{align*}
Such $\set{\nu_s}_{s<\log T}$ is called the \emph{associated rescaled Brakke flow}. 

We will prove a forward monotonicity formula for rescaled Brakke flows. To achieve this goal, we first introduce a useful cut-off function on space-time.

\begin{lem} \label{CutoffLem}
Consider the cut-off function
$$
\phi_R(p,s)=\left(1-R^{-2}e^s(|\mathbf{x}(p)|^2+2n)\right)^5_+.
$$
Fix any real numbers $\bar{s}_0<\bar{s}_1$. The following is true:
\begin{enumerate}
\item $\lim_{R\to\infty} \phi_R=1$ uniformly on compact subsets;
\item There is a constant $\hat{M}_0=\hat{M}_0(n,\bar{s}_0,\bar{s}_1)$ so that 
$$
\sup_{\bar{s}_0\leq s \leq \bar{s}_1} \Vert\nabla\phi_R(\cdot,s)\Vert_{C^1}+\Vert (\partial_s-\mathscr{L}) \phi_R(\cdot,s)\Vert_{C^0} \leq \hat{M}_0 R^{-1}
$$
where $\mathscr{L}=\Delta+\frac{\mathbf{x}}{2}\cdot\nabla$;
\item There is a constant $\hat{M}_1=\hat{M}_1(n,\bar{s}_0,\bar{s}_1)$ so that, for all $\bar{s}_0\leq s\leq\bar{s}_1$,
$$
 \Vert \phi_R (\cdot,s)\Vert_{C^{3}}+\Vert\partial_s \phi_R (\cdot,s)\Vert_{C^1}+\sum_{i=1}^2\Vert(1+|\mathbf{x}|)\nabla^i \phi_R (\cdot,s)\Vert_{C^0} \leq \hat{M}_1.
$$
\end{enumerate}
\end{lem}

\begin{proof}
The first claim follows from the definition of $\phi_R$. The second and third claim can be checked by straightforward, but tedious, computations, so we omit the details. 
\end{proof}

\begin{prop} \label{MonotoneProp}
Let $\set{\mu_t}_{t\in (0,T)}$ be an integral Brakke flow that satisfies
\begin{enumerate}
\item $\lim_{t\to 0} \mu_t=\mathcal{H}^n\lfloor\mathcal{C}$;
\item For every $t\in (0,T)$, $t^{-\frac{1}{2}}\mathrm{spt}(\mu_t)\subseteq\overline{\Omega^\prime}$.
\end{enumerate}
Let $\set{\nu_s}_{s<\log T}$ be the associated rescaled flow. There is a constant $E_0=E_0(\Gamma_0,\Omega^\prime,\cC)$ so that, for all $s<\log T$, 
$$
E_{rel}[\nu_s,\Gamma_0]=\lim_{R\to\infty}\left(\int_{\bar{B}_R} e^{\frac{|\mathbf{x}|^2}{4}} \, d\nu_s-\int_{\bar{B}_R\cap \Gamma_0} e^{\frac{|\mathbf{x}|^2}{4}} \, d\mathcal{H}^n\right) $$
exists and is bounded by $E_0$. Moreover, for any $-\infty<\bar{s}_0<\bar{s}_1<\log T$, if $f\geq 0$ satisfies
$$
\hat{M}=\sup_{\bar{s}_0\leq s\leq\bar{s}_1} \Vert f(\cdot,s)\Vert_{C^{3}}+\Vert\partial_s f(\cdot,s)\Vert_{C^1}+\sum_{i=1}^2\Vert(1+|\mathbf{x}|)\nabla^i f(\cdot,s)\Vert_{C^0}<\infty,
$$
then, for all $\bar{s}_0\leq s_0\leq s_1\leq\bar{s}_1$,
\begin{equation}\label{ErelMonIneq}
\begin{split}
 E_{rel}[\nu_{s_0},\Gamma_0; f] & \geq E_{rel}[\nu_{s_1},\Gamma_0; f] 
 +\int_{s_0}^{s_1} \int f \left|\mathbf{H}-\frac{\mathbf{x}}{2}\cdot S^\perp\right|^2 e^{\frac{|\mathbf{x}|^2}{4}}\, d\nu_s ds \\  & -\int_{s_0}^{s_1} E_{rel}\left[\nu_s,\Gamma_0; (\partial_s-\mathscr{L}) f+Q_{\nabla^2 f}\right] ds.
\end{split}
\end{equation}
Here $S=S(\mathbf{x})=T_{\mathbf{x}}\nu_s$ and $Q_{\nabla^2 f}(p,\mathbf{v})=\nabla^2f(p,s)(\mathbf{v},\mathbf{v}  )$.
\end{prop}

\begin{rem}
When $\set{\mu_t}_{t\in (0,T)}$ is a smooth MCF there is an equality in \eqref{ErelMonIneq}.
\end{rem}

\begin{proof}[Proof of Proposition \ref{MonotoneProp}]
By our hypotheses, it follows from the pseudo-locality result \cite[Theorem 1.5]{IlmanenNevesSchulze} and interior regularity for mean curvature flow \cite{EHInterior} (cf. \cite[Proposition 3.3]{BWProperness}) that there are sufficiently large constants $\hat{R}_0=\hat{R}_0(\mathcal{C})$ and $\hat{C}_0=\hat{C}_0(\mathcal{C})$ so that, for every $s<\log T$, there is an asymptotically conical hypersurface $\Gamma_s\in \mathcal{H}(\Gamma_0^\prime\setminus\bar{B}_{\hat{R}_0},\Gamma_1^\prime\setminus\bar{B}_{\hat{R}_0})$ that satisfies 
$$
\sup_{p\in\Gamma_s} |\mathbf{x}(p)| |A_{\Gamma_s}(p)| \leq \hat{C}_0
\mbox { and }
\nu_s\lfloor\mathbb{R}^{n+1}\setminus\bar{B}_{\hat{R}_0}=\mathcal{H}^n\lfloor\Gamma_s.
$$
It follows from Proposition \ref{RelEntropyAnnuliProp} and the dominated convergence theorem that, for any $R_2>R_1>\hat{R}_2$, 
\begin{align*}
\left| \int_{\bar{A}_{R_1, R_2}} e^{\frac{|\mathbf{x}|^2}{4}} \, d\nu_s -\int_{\bar{A}_{R_1,R_2}\cap \Gamma_0 } e^{\frac{|\mathbf{x}|^2}{4}} \, d\mathcal{H}^n \right|=\lim_{\delta \to 0} \left| E_{rel}[\Gamma_s, \Gamma_0; \alpha_{R_2,R_1, \delta}] \right|\leq \hat{C}_1 R_1^{-2}
\end{align*}
where $\hat{R}_2$ and $\hat{C}_1$ both depend only on $\Gamma_0$, $\Omega^\prime$ and $\cC$. It follows immediately that 
$$
 E_{rel}[\nu_s, \Gamma_0]=\lim_{R\to \infty} \left(  \int_{\bar{B}_R} e^{\frac{|\mathbf{x}|^2}{4}} \, d\nu_s -\int_{\bar{B}_R\cap \Gamma_0 } e^{\frac{|\mathbf{x}|^2}{4}} \, d\mathcal{H}^n \right)
$$
exists and is finite. By Huisken's monotonicity formula \cite{Huisken}, for all $s<\log T$ and all $R>1$,
$$
\nu_s(B_R) \leq K_0 R^n
$$
where $K_0=K_0(\mathcal{C})>0$ and so
$$
\left|  \int_{B_{2\hat{R}_2}} e^{\frac{|\mathbf{x}|^2}{4}} \, d\nu_s -\int_{B_{2\hat{R}_2}\cap \Gamma_0 } e^{\frac{|\mathbf{x}|^2}{4}} \, d\mathcal{H}^n\right| \leq E_1
$$
where $E_1=E_1(\Gamma_0,\hat{R}_2,K_0)$, in turn, depends only on $\Gamma_0,\Omega^\prime$ and $\cC$. Hence, by the triangle inequality and the two bounds already established, for any $R>2\hat{R}_2$,
$$
\left|\int_{\bar{B}_R} e^{\frac{|\mathbf{x}|^2}{4}} \, d\nu_s -\int_{\bar{B}_R\cap \Gamma_0 } e^{\frac{|\mathbf{x}|^2}{4}} \, d\mathcal{H}^n \right|\leq E_1+\frac{1}{2} \hat{C}_1 \hat{R}_2^{-2}
$$
and so the first claim follows with $E_0=E_1+\frac{1}{2} \hat{C}_1 \hat{R}_2^{-1}$ depending on $\Gamma_0,\Omega^\prime$ and $\cC$.

To prove the forward monotonicity formula, appealing to \cite[Section 3.5]{Brakke} and the divergence theorem, one computes
\begin{equation} \label{CutoffMonotoneEqn}
\begin{split}
E[\nu_{s_0},\Gamma_0; \phi_R f] & \geq E[{\nu_{s_1},\Gamma_0; \phi_R f}]+\int_{s_0}^{s_1} \int \phi_R f \left|\mathbf{H}-\frac{\mathbf{x}}{2}\cdot S^\perp\right|^2 e^{\frac{|\mathbf{x}|^2}{4}} \, d\nu_s ds \\
& -\int_{s_0}^{s_1} E[\nu_s,\Gamma_0; \zeta_R] \, ds
\end{split}
\end{equation}
where 
\begin{align*}
\zeta_R &= \phi_R \left(\partial_s-\mathscr{L}\right) f+\phi_R Q_{\nabla^2f} +f \left(\partial_s-\mathscr{L}\right)\phi_R+fQ_{\nabla^2\phi_R}\\
&-2\nabla\phi_R\cdot\nabla f+2Q_{\nabla\phi_R (\nabla f)^T}\in C^0_c(\Real^{n+1}\times \mathbb{S}^n\times[\bar{s}_0,\bar{s}_1]).
\end{align*}

The hypotheses on $f$ and Lemma \ref{CutoffLem} ensure that $\zeta_R(\cdot,s)\in \mathfrak{X}^{e}(\Real^{n+1})$ and, moreover, $\Vert \zeta_R(\cdot,s) \Vert_{\mathfrak{X}}$ has a uniform (in $s$ and $R$) bound in terms of $n, \hat{M}$ and $\hat{M}_1$. The hypotheses on $f$ and Lemma \ref{CutoffLem} further imply that, for each fixed $s$,
$$
\lim_{R\to 0} \zeta_R = \left( \partial_s -\mathscr{L}\right)f+Q_{\nabla^2f}  \mbox{ uniformly on compact subsets.}
$$
By linearity,
$$
E_{rel}[\nu_s,\Gamma_0;\zeta_R]=E_{rel}[\Gamma_s,\Gamma_0; (1-\phi_{2\hat{R}_2,\delta})\zeta_R]+E[\nu_s,\Gamma_0; \phi_{2\hat{R}_2,\delta} \zeta_R].
$$
As $ \phi_{2\hat{R}_2,\delta} \zeta_R$  has compact support, the uniform convergence implies
$$
\lim_{R\to \infty} E[\nu_s,\Gamma_0; \phi_{2\hat{R}_2,\delta} \zeta_R]=E[\nu_s,\Gamma_0; \phi_{2\hat{R}_2,\delta} \left( \partial_s -\mathscr{L}\right)f+Q_{\nabla^2f}]
$$
Likewise, as uniform convergence on compact sets implies pointwise convergence, Proposition \ref{DominatedConvProp} implies
$$
\lim_{R\to \infty} E[\nu_s,\Gamma_0; (1-\phi_{2\hat{R}_2,\delta} )\zeta_R]=E[\nu_s,\Gamma_0; (1-\phi_{2\hat{R}_2,\delta} )\left( \partial_s -\mathscr{L}\right)f+Q_{\nabla^2f}].
$$
Hence,
$$
\lim_{R\to \infty} E[\nu_s,\Gamma_0;  \zeta_R]=E[\nu_s,\Gamma_0;  \left( \partial_s -\mathscr{L}\right)f+Q_{\nabla^2f}]
$$
Finally, by (suitably modifying) Theorem \ref{WeightRelThm}, one has
$$
\left|E[\nu_s,\Gamma_0;  \zeta_R]\right| \leq C_9(1+E_0)\Vert  \zeta_R\Vert_{\mathfrak{X}}
$$
is uniformly bounded  on compact intervals of time. Hence, by the dominated convergence theorem,
$$
\lim_{R\to\infty} \int_{s_0}^{s_1} E[\nu_s,\Gamma_0; \zeta_R] \, ds=\int_{s_0}^{s_1} E[\nu_s,\Gamma_0; \left(\partial_s-\mathscr{L}\right)f+Q_{\nabla^2f}] \, ds.
$$

Similarly, for each fixed $s$, as $\lim_{R\to\infty} \phi_R f=f$ pointwise and $\Vert\phi_Rf(\cdot,s)\Vert_{Lip}$ has a uniform (in $R$) bound, it follows from Proposition \ref{DominatedConvProp} and the dominated convergence theorem that 
\begin{align*}
\lim_{R\to\infty} E[\nu_s,\Gamma_0; \phi_R f] & =\lim_{R\to\infty} \left(E[\Gamma_s,\Gamma_0; (1-\phi_{2\hat{R}_2,\delta})\phi_R f]+E[\nu_s,\Gamma_0; \phi_{\hat{R}_2,\delta} \phi_R f] \right)\\
&=E[\Gamma_s,\Gamma_0; (1-\phi_{2\hat{R}_2,\delta})f]+E[\nu_s,\Gamma_0; \phi_{\hat{R}_2,\delta} f]=E[\nu_s,\Gamma_0; f].
\end{align*} 

Therefore, \eqref{ErelMonIneq} follows from \eqref{CutoffMonotoneEqn} by sending $R\to\infty$ and the monotone convergence theorem.
\end{proof}

We are now ready to prove Theorem \ref{WeightMonotoneThm}.

\begin{proof}[Proof of Theorem \ref{WeightMonotoneThm}]
Let $\set{\nu_s}_{s<\log T}$ be the associated rescaled Brakke flow. By Proposition \ref{MonotoneProp} with $f\equiv 1$, 
$$
\lim_{s_i\to -\infty} \int_{-\infty}^{s_i} \int \left|\mathbf{H}-\frac{\mathbf{x}}{2}\cdot S^\perp\right|^2 e^{\frac{|\mathbf{x}|^2}{4}} \, d\nu_s ds=0.
$$
Let $\nu_s^i=\nu_{s+s_i}$ and so each $\set{\nu^i_{s}}_{s<\log T-s_i}$ is an integral rescaled Brakke flow. By the area estimates and Brakke's compactness theorem, \cite{Brakke} or \cite[Section 7]{IlmanenElliptic}, there is a  subsequence $i_j\to\infty$ so that
$$
\set{\nu_s^{i_j}}_{s<\log T-s_{i_j}}\to\set{\hat{\nu}_s}_{s\in\mathbb{R}}
$$
as rescaled flows. It is not hard to see that, for any $s$,
$$
E_{rel}[\hat{\nu}_s,\Gamma_0]=E_{-\infty}
$$
and, for any $s_0\leq s_1$,
$$
\int_{s_0}^{s_1} \int \left|\mathbf{H}-\frac{\mathbf{x}}{2}\cdot S^\perp\right|^2 e^{\frac{|\mathbf{x}|^2}{4}} \, d\hat{\nu}_s ds=0.
$$
In particular, for a.e. $s$, $\hat{\nu}_s$ is a critical point for the functional $E$. This implies $\hat{\nu}_s=\hat{\nu}$ is static and, as $\spt(\nu^{i_j}_s)\subseteq\overline{\Omega^\prime}$, it follows that $\spt(\hat{\nu})\subseteq\overline{\Omega^\prime}$. Finally, as observed in the proof of Proposition \ref{MonotoneProp}, the $\nu_s$ are $C^1$-asymptotic to $\mathcal{C}$ in a uniform manner and so $\hat{\nu}$ is also asymptotic to $\mathcal{C}$. The claim follows from this by unwinding the construction of $\nu_s^{i_j}$.
\end{proof}

\appendix

\section{Auxiliary lemmas} \label{AuxiliaryApp}
For a hypersurface $\Sigma$, let 
$$
\mathscr{L}^\mu_\Sigma=\Delta_\Sigma+\frac{\mathbf{x}}{2}\cdot\nabla_\Sigma-\mu
$$
and when $\mu=\frac{1}{2}$ we write $\mathscr{L}^{\frac{1}{2}}_\Sigma=\mathscr{L}_\Sigma$. We then let
$$
L_\Sigma=\mathscr{L}_\Sigma+|A_\Sigma|^2=\Delta_\Sigma+\frac{\mathbf{x}}{2}\cdot\nabla_\Sigma-\frac{1}{2}+|A_\Sigma|^2.
$$

\begin{lem} \label{BarrierEqnLem}
If $\Sigma$ is a $C^2$-asymptotically conical self-expanding end in $\mathbb{R}^{n+1}$, then
$$
\mathscr{L}^0_\Sigma \left(r^d e^{-\frac{r^2}{4}}\right)=-\frac{1}{2}\left(n+d+O(r^{-2})\right)r^d e^{-\frac{r^2}{4}}
$$
where $r(p)=|\mathbf{x}(p)|$ for $p\in\Sigma$.
\end{lem}

\begin{proof}
By the chain rule
$$
\nabla_\Sigma \left(r^d e^{-\frac{r^2}{4}}\right)=\left(\frac{d}{r}-\frac{r}{2}\right)r^d e^{-\frac{r^2}{4}} \nabla_\Sigma r
$$
and
$$
\Delta_\Sigma \left(r^d e^{-\frac{r^2}{4}}\right)=\left\{\left(\frac{r^2}{4}-d-\frac{1}{2}+\frac{d^2-d}{r^2}\right)|\nabla_\Sigma r|^2+\left(\frac{d}{r}-\frac{r}{2}\right) \Delta_\Sigma r\right\} r^d e^{-\frac{r^2}{4}}.
$$
Thus, combining these gives
$$
\mathscr{L}_\Sigma^0 \left(r^d e^{-\frac{r^2}{4}}\right)=\left\{\left(-\frac{d+1}{2}+\frac{d^2-d}{r^2}\right)|\nabla_\Sigma r|^2+\left(\frac{d}{r}-\frac{r}{2}\right) \Delta_\Sigma r\right\} r^d e^{-\frac{r^2}{4}}.
$$
Observe that by our hypotheses on $\Sigma$
$$
|\nabla_\Sigma r|^2=1+O(r^{-4}) \mbox{ and } \Delta_\Sigma r=\frac{n-1}{r}+O(r^{-3}).
$$
Hence, 
$$
\mathscr{L}_\Sigma^0 \left(r^d e^{-\frac{r^2}{4}}\right)=-\frac{1}{2}\left(n+d+O(r^{-2})\right) r^d e^{-\frac{r^2}{4}},
$$
proving the claim.
\end{proof}

\begin{lem} \label{ExpanderMeanCurvLem}
Fix an $\bar{M}_0>1$ and suppose $\Sigma$ is a self-expander in an open subset of $\mathbb{R}^{n+1}$ with $\sup_{\Sigma}|A_\Sigma|+|\nabla_\Sigma A_\Sigma|\leq \bar{M}_0$. If $v\in C^2(\Sigma)$ with $\Vert v \Vert_{C^2}\leq (8\bar{M}_0)^{-1}$ is such that $\mathbf{h}=\mathbf{x}|_\Sigma+v\mathbf{n}_\Sigma$ is a $C^{2}$ embedding, then at $p\in \Sigma$
$$
H_{\mathbf{h}(\Sigma)}+\frac{\mathbf{x}}{2}\cdot\mathbf{n}_{\mathbf{h}(\Sigma)}=-L_\Sigma v+Q(v,\mathbf{x}\cdot \nabla_\Sigma v,\nabla_\Sigma v,\nabla_\Sigma^2 v)
$$
where $Q$ depends on $p,v$, and $\Sigma$ and is a homogeneous degree-two polynomial of the form
$$
Q(s,\rho,\mathbf{d},\mathbf{T})=\mathbf{a}(s, \rho, \mathbf{d}, \mathbf{T})\cdot \mathbf{d} + b(s, \mathbf{d}, \mathbf{T}) s.
$$
Here $\mathbf{a}$ and $b$ are homogeneous degree-one polynomials with coefficients bounded by $\bar{C}_1=\bar{C}_1(n,\bar{M}_0).$
\end{lem}

\begin{proof}
Denote by $\Gamma=\mathbf{h}(\Sigma)$. First, by \cite[Lemma 7.2]{BWBanach},
\begin{equation} \label{ExpanderMeanCurvEqn}
H_{\Gamma}+\frac{\mathbf{x}}{2}\cdot\mathbf{n}_\Gamma=-\left(\mathscr{L}_\Sigma(v\mathbf{n}_\Sigma)+\sum_{i,j=1}^n (g_{\mathbf{h}}^{-1}-g_\Sigma^{-1})^{ij}(\nabla_\Sigma^2 \mathbf{h})_{ij}\right)\cdot(\mathbf{n}_\Gamma\circ\mathbf{h})
\end{equation}
where $g_\mathbf{h}$ and $g_\Sigma$ are the pull-back metrics of the Euclidean one via $\mathbf{h}$ and $\mathbf{x}|_\Sigma$, respectively, and we used the fact $\mathscr{L}_\Sigma\mathbf{x}=\mathbf{0}$. One readily computes that
$$
(g_{\mathbf{h}})_{ij}=(g_\Sigma)_{ij} +\partial_i v \partial_j v+2v (A_\Sigma)_{ij}+v^2 \sum_{k=1}^n (A_{\Sigma})_{ik}(A_\Sigma)^{k}_j
$$
and so the hypotheses ensure 
$$
2g_{\Sigma}> g_\mathbf{h}>\frac{1}{2} g_\Sigma.
$$
Using this, a direct computation gives 
\begin{align*}
(g_\mathbf{h}^{-1}-g_\Sigma^{-1})^{ij}&=
-2A_{\Sigma}^{ij} v +Q_1^{ij} (v,\nabla_\Sigma v)
\end{align*}
where $Q_1$ is a homogeneous degree-two polynomial valued in $(2,0)$-tensors and  of the form
$$
Q_1(v, \nabla_\Sigma v)=\mathbf{a}_1(\nabla_\Sigma v) \cdot \nabla_\Sigma v + b_1( v) v
$$ 
where $\mathbf{a}_1$ and $b_1$ are homogeneous degree-one polynomials valued in $(2,0)$-tensors and with coefficients bounded by $K_1=K_1(n,\bar{M}_0)$. Likewise,
$$
\nabla_\Sigma^2 \mathbf{h}=\nabla_\Sigma^2 \mathbf{x}|_{\Sigma}+\mathbf{Q}_2( v, \nabla_\Sigma v, \nabla^2_\Sigma v) 
$$
where $\mathbf{Q}_2$ is a degree-one polynomial with coefficients bounded by $K_2=K_2(\bar{M}_0)$ and valued in vector-valued symmetric $(0,2)$-tensors. Finally, 
$$
\mathbf{n}_\Gamma\circ\mathbf{h}=\mathbf{n}_\Sigma+\mathbf{Q}_3 (\nabla_\Sigma v)
$$
where $\mathbf{Q}_3$ is a vector-valued homogeneous degree-one polynomial of the form
$$
\mathbf{Q}_3(\nabla_\Sigma v)= \mathbf{a}_3(\nabla_\Sigma v)+\mathbf{a}_3^\prime(\nabla_\Sigma v) \mathbf{n}_\Sigma.
$$
Here $\mathbf{a}_3$ and $\mathbf{a}_3^\prime$ have coefficients bounded by $K_3=K_3(n,\bar{M}_0)$ and $\mathbf{a}_3\cdot \mathbf{n}_\Sigma=0$.

By \cite[Lemma 5.9]{BWBanach}, on a self-expander 
$$
\mathscr{L}_\Sigma^0\mathbf{n}_\Sigma+|A_\Sigma|^2\mathbf{n}_\Sigma=\mathbf{0}.
$$
Using this, one obtains
\begin{align*}
\mathscr{L}_\Sigma(v\mathbf{n}_\Sigma)\cdot (\mathbf{n}_\Gamma\circ\mathbf{h}) &= \mathscr{L}_\Sigma(v\mathbf{n}_\Sigma)\cdot\mathbf{n}_\Sigma+\mathscr{L}_\Sigma(v\mathbf{n}_\Sigma)\cdot \mathbf{Q}_3(\nabla_\Sigma v)\\
&=\mathscr{L}_\Sigma(v\mathbf{n}_\Sigma)\cdot\mathbf{n}_\Sigma+Q_4(v, \nabla_\Sigma v, \mathbf{x}\cdot\nabla_\Sigma v, \nabla^2_\Sigma v)
\end{align*}
where $Q_4$ is a homogeneous degree-two polynomial of the form
$$
Q_4(v, \nabla_\Sigma v, \mathbf{x}\cdot\nabla_\Sigma v, \nabla^2_\Sigma v)=\mathbf{a}_4(v,  \mathbf{x}\cdot\nabla_\Sigma v, \nabla_\Sigma v,\nabla^2_\Sigma v)\cdot \nabla_\Sigma v.
$$
Moreover, the coefficients of $\mathbf{a}_4$ are bounded by $K_4=K_4(n,\bar{M}_0)$. Similarly, as $\mathbf{n}_\Sigma\cdot (\nabla_{\Sigma}^2 \mathbf{x})_{ij}=(A_\Sigma)_{ij}$,
\begin{align*}
(\mathbf{n}_\Gamma\circ\mathbf{h})\cdot \sum_{i,j=1}^n (g_{\mathbf{h}}^{-1}-g_\Sigma^{-1})^{ij}(\nabla_\Sigma^2 \mathbf{h})_{ij} &=2v \mathbf{n}_\Sigma\cdot\sum_{i,j=1}^n -A_\Sigma^{ij} (\nabla_\Sigma^2 \mathbf{x}|_{\Sigma})_{ij} +Q_5(v, \nabla_\Sigma v, \nabla_\Sigma^2 v)\\
&=2 |A_\Sigma|^2 v +Q_5( v, \nabla_\Sigma v, \nabla_\Sigma^2 v)
\end{align*}
where $Q_5(v, \nabla_\Sigma v, \nabla_\Sigma^2 v)$ is a homogeneous degree-two polynomial of the form
$$
Q_5(v, \nabla_\Sigma v, \nabla_\Sigma^2 v) =\mathbf{a}_5(v, \nabla_\Sigma v, \nabla_\Sigma^2 v)\cdot \nabla_\Sigma v+ b_5(v, \nabla_\Sigma v, \nabla_\Sigma^2 v) v
$$
and the coefficients of $\mathbf{a}_5$ and of $b_5$ are bounded by $K_5=K_5(n,\bar{M}_0)$.

Hence, substituting these into the above expressions into the formula \eqref{ExpanderMeanCurvEqn} gives 
\begin{align*}
H_\Gamma+\frac{\mathbf{x}}{2}\cdot\mathbf{n}_\Gamma = -\left(\mathscr{L}_\Sigma(v\mathbf{n}_\Sigma)\cdot\mathbf{n}_\Sigma+2 |A_\Sigma|^2 v \right)+Q( v,  \mathbf{x}\cdot \nabla_\Gamma v,\nabla_\Sigma v , \nabla_\Sigma^2 v).
\end{align*}
Here $Q=-Q_4-Q_5$, and so is of the form desired and with coefficients bounded by $\bar{C}_1=\bar{C}_1(n,\bar{M}_0)$.

Finally, we compute
$$
\mathscr{L}_\Sigma(v\mathbf{n}_\Sigma)\cdot\mathbf{n}_\Sigma =\mathscr{L}_\Sigma v+v\mathbf{n}_\Sigma\cdot\mathscr{L}^0_\Sigma\mathbf{n}_\Sigma+2(\nabla_\Sigma v\cdot\nabla_\Sigma\mathbf{n}_\Sigma)\cdot\mathbf{n}_\Sigma
=\mathscr{L}_\Sigma v-|A_\Sigma|^2 v,
$$
which completes the proof.
\end{proof}

\section{Geometric computations}

\begin{prop}\label{VolumeEstProp}
Let $\sigma$ be a $C^2$-hypersurface in $\mathbb{S}^n$ with unit normal $\nu_\sigma$ and assume that
$$
K_\sigma=\sup_{p\in\sigma} |A_\sigma(p)|<\infty.
$$
There is a constant $\delta_0=\delta_0(K_\sigma,n)\in (0,1)$ so that if $\theta\colon\sigma\to (0,\frac{\pi}{2})$ satisfies $\Vert\theta\Vert_1<\delta_0$, then the set 
$$
\omega=\set{\cos(t\theta(p))\mathbf{x}(p)+\sin(t\theta(p))\nu_\sigma (p)\colon 0<t<1, p\in\sigma}
$$ 
is an open domain in $\mathbb{S}^n$ with the volume estimate
$$
\mathcal{H}^n(\omega)\leq 2\int_{\sigma} \theta\, d\mathcal{H}^{n-1}.
$$
\end{prop}

\begin{proof}
Fix any point $p\in\sigma$. Let $\phi^{-1}$ be the normal coordinates on an open neighborhood of $p$ in $\sigma$; i.e., $\phi\colon B^{n-1}_\epsilon\to \sigma\subset\mathbb{R}^{n+1}$ is a $C^2$ diffeomorphism onto its image so that $\phi(\mathbf{0})=p$ and, for $1\leq i,j \leq n-1$,
$$
\partial_{x_i} \phi(\mathbf{0})\cdot\partial_{x_j}\phi(\mathbf{0})=\delta_{ij} \mbox{ and } \nabla_{\sigma}\nu_\sigma (p)\cdot\partial_{x_i}\phi(\mathbf{0})=\kappa_i \partial_{x_i}\phi(\mathbf{0})
$$
where the $\kappa_i$ are principle curvatures of $\sigma$ at $p$. Write $\theta(x)=\theta(\phi(x))$ and $\nu_\sigma(x)=\nu_\sigma(\phi(x))$. Define
$$
\mathbf{f}(t,x)=\cos(t\theta(x))\phi(x)+\sin(t\theta(x))\nu_\sigma(x).
$$
	
Next we compute $\mathbf{f}^* dvol_{\mathbb{S}^n}(t,\mathbf{0})$. A straightforward computation gives that
\begin{align*}
\partial_t\mathbf{f} (t,\mathbf{0}) &= -\sin(t\theta(\mathbf{0}))\theta(\mathbf{0})\phi(\mathbf{0})+\cos(t\theta(\mathbf{0}))\theta(\mathbf{0})\nu_\sigma(\mathbf{0}); \mbox{ and;} \\
\partial_{x_i}\mathbf{f} (t,\mathbf{0}) & =-t\sin(t\theta(\mathbf{0})) \partial_{x_i}\theta(\mathbf{0})\phi(\mathbf{0})+t\cos(t\theta(\mathbf{0}))\partial_{x_i}\theta(\mathbf{0})\nu_\sigma(\mathbf{0}) \\
& \quad +(\cos(t\theta(\mathbf{0}))+\kappa_i\sin(t\theta(\mathbf{0})))\partial_{x_i}\phi(\mathbf{0}).
\end{align*}
It follows that
\begin{align*}
\partial_t\mathbf{f}(t,\mathbf{0})\cdot\partial_t\mathbf{f}(t,\mathbf{0}) & = \theta^2(\mathbf{0}); \\
\partial_t\mathbf{f}(t,\mathbf{0})\cdot\partial_{x_i}\mathbf{f}(t,\mathbf{0}) & = t\theta(\mathbf{0})\partial_{x_i}\theta(\mathbf{0}); \mbox{ and;}\\
\partial_{x_i}\mathbf{f}(t,\mathbf{0})\cdot\partial_{x_j}\mathbf{f}(t,\mathbf{0}) & = \delta_{ij}+(\kappa_i+\kappa_j)\cos(t\theta(\mathbf{0}))\sin(t\theta(\mathbf{0}))\delta_{ij}\\
& +(\kappa_i\kappa_j-1)\sin^2(t\theta(\mathbf{0})))\delta_{ij}+t^2\partial_{x_i}\theta(\mathbf{0})\partial_{x_j}\theta(\mathbf{0}),
\end{align*}
where we used the fact that $|\phi|=|\nu|=1$ and $\phi\cdot\nu=\partial_{x_i}\phi\cdot\nu=0$. Hence, if $\delta_0$ is chosen sufficiently small, then $|\sin(t\theta(\mathbf{0}))|\leq t\theta(\mathbf{0})$ and 
$$
0<\mathbf{f}^* dvol_{\mathbb{S}^n}(t,\mathbf{0}) \leq 2\theta(\mathbf{0})\, dxdt.
$$
In particular, $\mathbf{f}$ is a $C^1$ diffeomorphism from $(0,1)\times\sigma$ onto its image and so the set $\omega$ is an open domain in $\mathbb{S}^n$.
\end{proof}

\end{document}